\documentclass[12pt, reqno, a4paper, oneside]{amsart}
\usepackage{graphicx, epsfig, psfrag, subfigure}
\usepackage{amsfonts,amsmath,amssymb,amsbsy,amsthm}
\usepackage{bm}
\usepackage{color }
\usepackage{float}
\usepackage{mathrsfs}
\usepackage[left=3cm,right=3cm,top=3cm,bottom=3cm]{geometry}
\usepackage{longtable}
\usepackage{todonotes}
\usepackage{epsfig}

\usepackage{environments}
\numberwithin{theorem}{section}
\numberwithin{equation}{section}

\renewcommand{\cases}[1]{\left\{ \begin{array}{rl} #1 \end{array} \right.}

\def\XXint#1#2#3{{\setbox0=\hbox{$#1{#2#3}{\int}$ }
\vcenter{\hbox{$#2#3$ }}\kern-.6\wd0}}

\usepackage{accents}
\newlength{\dhatheight}
\newlength{\dtildeheight}

\def\b{\big}

\def\sep{\,|\,}
\def\bsep{\,\b|\,}

\def\diam{{\rm diam}}

\def\supp{{\rm supp}}

\def\R{\mathbb{R}}

\def\Z{\mathbb{Z}}

\def\dx{\,{\rm d}x}
\def\dy{\,{\rm d}y}

\def\dr{\,{\rm d}r}
\def\dt{\,{\rm d}t}

\def\pp{\partial}

\def\<{\langle}
\def\>{\rangle}

\def\mA{{\sf A}}

\def\mF{{\sf F}}

\def\bfa{{\bm a}}
\def\bfg{{\bm g}}

\def\bfh{{\bm h}}

\def\D{\nabla}
\def\del{\delta}
\def\ddel{\delta^2}

\def\a{{\rm a}}
\def\c{{\rm c}}
\def\ac{{\rm ac}}
\def\i{{\rm i}}

\def\L{\Lambda}

\def\Is{\mathcal{I}}

\def\As{\mathcal{A}}
\def\Cs{\mathcal{C}}

\def\E{\mathcal{E}}

\def\tilu{\tilde u}
\def\tilv{\tilde v}

\def\T{\mathcal{T}}

\definecolor{cocol}{rgb}{0.7, 0, 0}
\definecolor{hwcol}{rgb}{0, 0.7, 0}
\definecolor{adcol}{rgb}{0, 0, 0.7}

\begin{document}

\title[A/C coupling with higher-order finite elements]{Analysis of patch-test
  consistent atomistic-to-continuum coupling with higher-order finite elements}

\author{A. S. Dedner}
\address{A. Dedner\\ Mathematics Institute \\ Zeeman Building \\
 University of Warwick \\ Coventry CV4 7AL \\ UK}
\email{a.s.dedner@warwick.ac.uk}

\author{C. Ortner}
\address{C. Ortner\\ Mathematics Institute \\ Zeeman Building \\
 University of Warwick \\ Coventry CV4 7AL \\ UK}
\email{c.ortner@warwick.ac.uk}

\author{H. Wu}
\address{H. Wu\\ Mathematics Institute \\ Zeeman Building \\
 University of Warwick \\ Coventry CV4 7AL \\ UK}
\email{huan.wu@warwick.ac.uk}

\date{\today}
\thanks{HW was supported by MASDOC doctoral training centre, EPSRC grant
  EP/H023364/1. CO was supported by ERC Starting Grant 335120.}
%
%


\begin{abstract}
  We formulate a patch test consistent atomistic-to-continuum coupling (a/c)
  scheme that employs a second-order (potentially higher-order) finite element
  method in the material bulk. We prove a sharp error estimate in the
  energy-norm, which demonstrates that this scheme is (quasi-)optimal amongst
  energy-based sharp-interface a/c schemes that employ the Cauchy--Born
  continuum model. Our analysis also shows that employing a higher-order
  continuum discretisation does not yield qualitative improvements to the rate
  of convergence.
\end{abstract}

\maketitle

\def\arraystretch{1.5}


\section{Introduction}
Atomistic-to-continuum (a/c) coupling is a class of coarse-graining methods for
efficient atomistic simulations of systems that couple localised atomistic
effects described by molecular mechanics with long-range elastic effects
described by continuum models using the finite-element method. We refer to
\cite{acta}, and references therein for an extensive introduction and
references.

The presented work explores the feasibility and effectiveness of introducing
higher-order finite element methods in the a/c framework, specifically for
quasi-nonlocal (QNL) type methods.

The QNL-type coupling, first introduced in \cite{Shimokawa:2004}, is an a/c
method that uses a ``geometric consistency condition'' \cite{E:2006} to
construct the coupling between the atomistic and continuum model. The first
explicit construction of such a scheme for two-dimensional domains with corners is
described in \cite{PRE-ac.2dcorners} for a neareast-neighbour many-body site
potential. We call this construction "G23" for future reference. This approach
satisfies force and energy patch tests (often simply called consistency),
which in particular imply absence of ghost forces.

We will supply the G23 scheme with finite element methods of different orders
and investigate the rates of convergence for the resulting schemes. Our
conclusion will be that second-order finite element schemes are theoretically
superior to first-order schemes, while schemes of third and higher order do not
improve the rate of convergence. This is due to the fact that the consistency
error of the a/c scheme is dominated by the modelling error committed at the a/c
interface. We will also explore, for some basic model problems, how well
second-order schemes fare in practise against first-order schemes.

\subsection{Outline}
\label{sec:intro:outline}
The theory of high-order finite element methods (FEM) in partial differential
equations, and applications in solid mechanics is well established; see
\cite{Schwab} and references therein. However, most work on the rigorous error
analysis of a/c coupling has been restricted to P1 finite element methods; the
only exception we are aware of is \cite{OrtnerZhang:SISC}, which focuses on
blending-type methods.

In the present work we estimate the accuracy of a QNL method employing a P2 FEM in
the continuum region against an exact solution obtain from a fully atomistic
model. Since stability of QNL type couplings is a subtle issue
\cite{2013-stab.ac} we will primarily analyse the consistency errors, taking
account the relative sizes of the fully resolved atomistic region and of the
entire computational domain (Sections 5.1-5.4). We will then optimize these
relative sizes as well as the mesh grading in the continuum region in order to
minimize the total consistency error (Section 5.5). We will observe that, using
P1-FEM in the continuum region, the error resulting from FEM approximations is
the dominating contributor of the consistency estimates, which implies that
increasing the order of the FEM can indeed improve the accuracy of the
simulation. We will show that, using Pk-FEM with $k\ge 2$, the FEM approximation
error is dominated by the interface error which comes purely from the G23
construction, and in particular demonstrate that the P2-FEM is sufficient to
achieve the optimal convergence rate for the consistency error. Finally, {\em
  assuming} the stability of G23 coupling (see \cite{2013-stab.ac} why this must
be an assumption), we prove a rigorous error estimate in \S \ref{sect:consist}.

Finally, we conduct numerical experiments on a 2D anti-plane model problem to
test our analytical predictions. The numerical results display the predicted
error convergence rates for the fully atomistic model, P1-FEM G23 model, and
P2-FEM G23 model.

\section{Preliminaries}
Our setup and notation follows \cite{PRE-ac.2dcorners}.  As our model geometry
we consider an infinite 2D triangular lattice,
\begin{equation*}
\L := \mA \Z^2, \quad \text{with } \mA =
\begin{pmatrix}
1 & \cos(\pi /3) \\ 0 & \sin(\pi/3)
\end{pmatrix}.
\end{equation*}
We define the six nearest-neighbour lattice directions by $a_1 :=(1,0)$, and
$a_j := Q_6^{j-1}a_1, j \in \Z$, where $Q_6$ denotes the rotation through the
angle $\pi/3$.  We supply $\L$ with an {\em atomistic triangulation}, as shown
in Figure~\ref{fig:lattice}, which will be convenient in both analysis and
numerical simulations. We denote this triangulation by $\T$ and its elements by
$T\in \T$. We also denote $\bfa := (a_j)_{j=1}^6$, and
$\mF \bfa : = (\mF a_j)_{j=1}^6$, for $\mF \in \R^{m\times 2}$.

\begin{figure}
\centering
\includegraphics[width=0.5\linewidth]{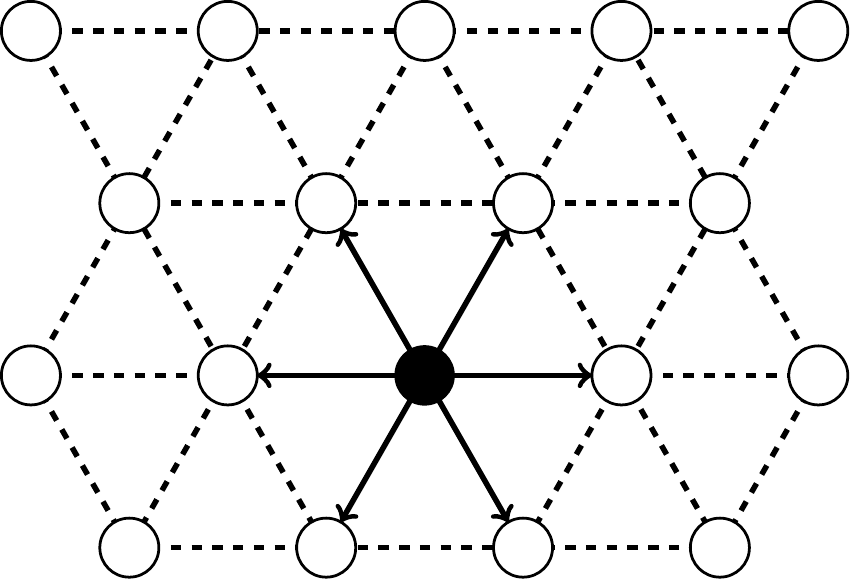}
\caption{\small{The lattice and its canonical triangulation.}}
\label{fig:lattice}
\end{figure}

We identify a discrete displacement map $u: \L \rightarrow \R^m $, $m = 1,2,3$, with its continuous piecewise affine interpolant, with
weak derivative $\D u$, which is also the pointwise derivative on each element
$T \in \mathcal{T}$. For $m = 1,2,3$, the spaces of displacements are defined as
\begin{equation*}
\begin{aligned}
\mathcal{U}_0 & := \b\{u \sep \L \rightarrow \R^m : \supp(\D u ) \text{ is compact} \b \}, \quad \text{and}\\
\dot{\mathcal{U}}^{1,2}  &:= \b\{u \sep \L \rightarrow \R^m : \D u  \in L^2\b\} .
\end{aligned}
\end{equation*}
We equip $\dot{\mathcal{U}}^{1,2}$ with the $H^1$-seminorm, $\|u\|_{\mathcal{U}^{1,2}} := \|\D u\|_{L^2(\R^2)}$. From \cite{OrShSu:2012} we know that $\mathcal{U}_0$ is dense in $\dot{\mathcal{U}}^{1,2}$ in the sense that, if $u \in \dot{\mathcal{U}}^{1,2}$, then there exist $ u_j \in \mathcal{U}_0$ such that $\D u_j \rightarrow \D u$ strongly in $L^2$.

A \emph{homogeneous displacement} is a map $u_\mF: \L \rightarrow \R^m, u_\mF(x) : = \mF x$, where $\mF\in \R^{m\times2}$.

For a map $u:\L \rightarrow \R^m$, we define the finite difference operator
\begin{equation}\label{def:diff_op}
\begin{aligned}
D_j u(x) &:= u(x+a_j)-u(x), \quad x \in \L, j \in \{1,2,...,6\}, \quad \text{and}\\
Du(x) &:= (D_j u(x))_{j=1}^6.
\end{aligned}
\end{equation}
Note that $Du_{\mF}(x) = \mF \bfa$.

\subsection{2D many-body nearest neighbour interactions}
\label{sec:energy}
We assume that the atomistic interaction is described by a nearest-neighbour
many-body site energy potential $V \in C^r(\R^{m\times 6})$,$r\ge 5$, with
$V(\mathbf{0}) = 0$. Furthermore, we assume that $V$ satisfies the \emph{point
  symmetry}
\begin{equation*}
V((-g_{j+3})_{j=1}^6) = V(\bfg) \quad \forall \bfg \in \R^{m\times 6}.
\end{equation*} The energy of a displacement $u\in \mathcal{U}_0$, given by
\begin{equation*}
\E^\a(u): = \sum_{\ell\in \Lambda}V(Du(\ell)),
\end{equation*}
is well-defined since the infinite sum becomes finite.  To formulate a variational problem in the energy space
$\dot{\mathcal{U}}^{1,2}$, we need the following lemma to extend $\E^\a$ to
$\dot{\mathcal{U}}^{1,2}$.

\begin{lemma}
  $\E^\a :(\mathcal{U}_0,\|\D \cdot\|_{L^2} )\rightarrow \R$ is continuous and
  has a unique continuous extension to $\dot{\mathcal{U}}^{1,2}$, which we still
  denote by $\E^\a $. Furthermore, the extended
  $\E^\a :(\dot{\mathcal{U}}^{1,2},\|\D \cdot\|_{L^2} ) \rightarrow \R$ is
  $r$-times continuously Fr\'{e}chet differentiable.
\end{lemma}
\begin{proof}
	See Lemma 2.1 in \cite{EhrOrtSha:2013}.
\end{proof}

We model a point defect by adding an external potential
$f \in C^r(\dot{\mathcal{U}}^{1,2})$ with $\partial_{u(\ell)} f(u) = 0$ for all
$|\ell| \geq R_f$, where $R_f$ is some given radius (the defect core radius),
and $f(u + c) = f(u)$ for all constants $c$. For example, we can think of $f$
modelling a substitutional impurity. See also \cite{2014-bqce, Or:2011a} for
similar approaches.

We then seek the solution to
\begin{equation}\label{eq:y_a}
u^\a \in \arg \min \b\{ \E^\a(u) - f(u)  \sep u \in \dot{\mathcal{U}}^{1,2} \b\}.
\end{equation}

For $u, \varphi,\psi \in \dot{\mathcal{U}}^{1,2}$ we define the \emph{first and
  second variations} of $\mathcal{E}^\a$ by
\begin{equation*}
\begin{aligned}
  \langle \delta \mathcal{E}^\a(u), \varphi \rangle &:= \lim_{t\rightarrow 0}t^{-1}\left(\mathcal{E}^\a (u+t\varphi)-\mathcal{E}^\a(u)\right), \\
  \<\del^2 \E^\a (u) \varphi, \psi\> &:=\lim_{t\rightarrow 0}t^{-1}\left(\<\del\mathcal{E}^\a (u+t\varphi),\psi\>-\<\del \mathcal{E}^\a(u), \psi\>\right).
\end{aligned}
\end{equation*}
We use analogous definitions for all energy functionals introduced in later sections.

\subsection{The Cauchy--Born Approximation}
The Cauchy--Born strain energy function, corresponding to the interatomic
potential $V$ is
\begin{equation*}
W(\mathsf{F}): = \frac{1}{\Omega_0} V(\mathsf{F}\bfa), \qquad \text{for }
\mathsf{F} \in \mathbb{R}^{m \times 2},
\end{equation*}
where $\Omega_0 := \sqrt{3}/2$ is the volume of a unit cell of the lattice $\Lambda$. Thus $W(\mathsf{F})$ is the energy per volume of the homogeneous lattice $\mathsf{F}\Lambda$.

\subsection{The G23 coupling method}
\label{sec:g23model}
Let $\As \subset \L$ denote the set of all lattices sites for which we want to maintain
full atomistic accuracy. We denote the set of interface lattice sites by
\begin{equation*}
 \Is := \b\{ \ell \in \L \setminus \As \bsep \ell +a_j \in \As \text{ for some } j \in \{1,\dots,6\} \b \}
\end{equation*}
and we denote the remaining lattice sites by
$\Cs : = \L \setminus (\As \cup \Is)$. Let $\Omega_\ell$ be the Voronoi cell
associated with site $\ell$. We define the continuum region
$\Omega^\c : = \R^2 \setminus \bigcup_{\ell \in \As \cup \Is} \Omega_{\ell}$;
see Figure \ref{fig:decomp}.

\begin{figure}
  \centering
  \includegraphics[width=0.65\linewidth]{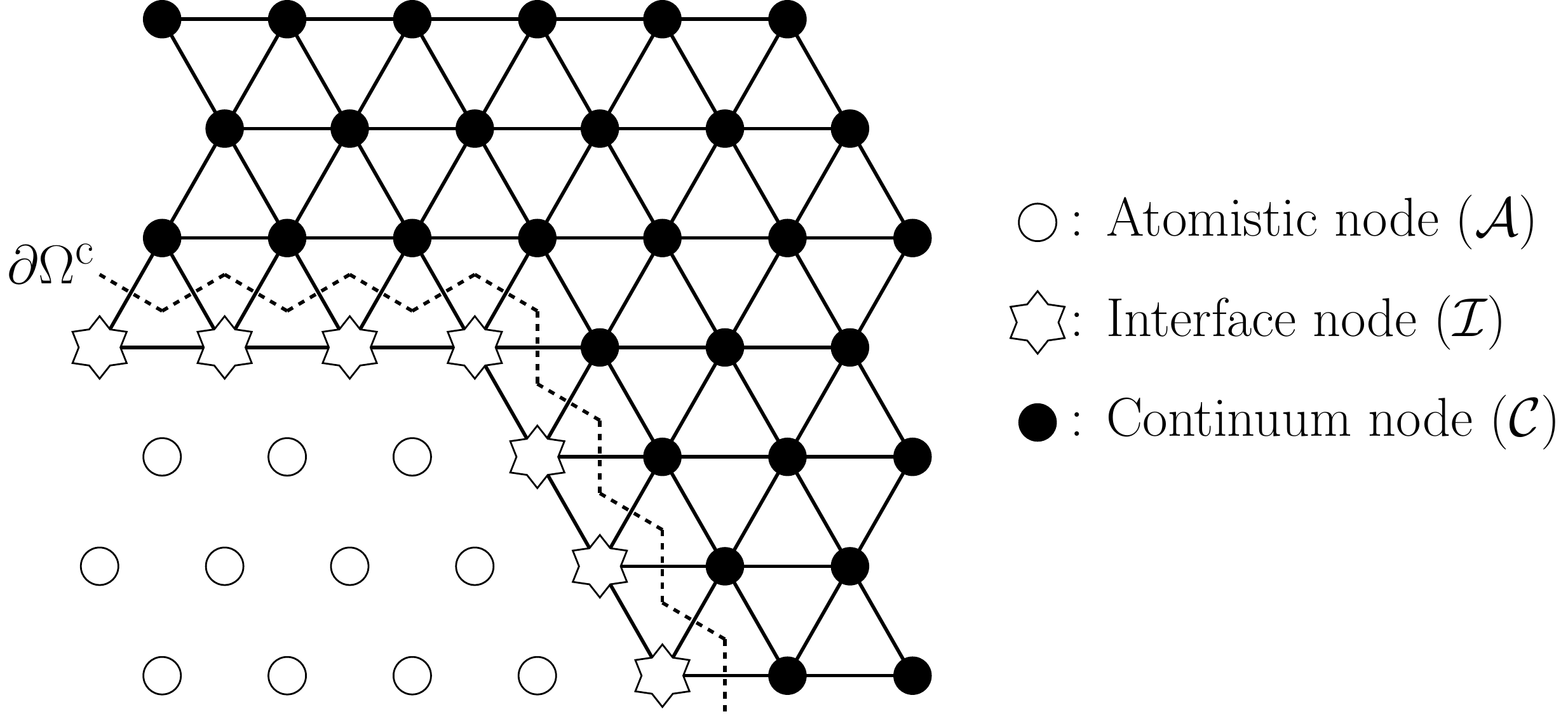}
  \caption{\small{The domain decomposition with a layer of interface atoms. }}
  \label{fig:decomp}
\end{figure}

A general form for the GRAC-type a/c coupling energy \cite{E:2006,
  PRE-ac.2dcorners} is
\begin{equation*}
\E^{\rm ac}(u) = \sum_{\ell\in \As} V(Du(\ell)) + \sum_{\ell\in \Is} V\left((\mathcal{R}_\ell D_ju(\ell))_{j=1}^6\right)+\int_{\Omega^\c} W(\D u(x)) \dx,
\end{equation*}
where $\mathcal{R}_\ell D_j u(\ell): = \sum_{i=1}^6C_{\ell,j,i}D_i u(\ell)$.
The parameters $C_{\ell,j,i}$ are to be determined in order for the coupling
scheme to satisfy the ``patch tests'':

$\mathcal{E}^{\rm ac}$ is \emph{locally energy consistent} if, for all $\mF \in \mathbb{R}^{m\times2}$,
\begin{equation}\label{econs}
V_\ell^i(\mF\bfa) = V(\mF\bfa) \quad \forall \ell \in \mathcal{I}.
\end{equation}

$\mathcal{E}^{\rm ac}$ is \emph{force consistent} if, for all $\mF \in \mathbb{R}^{m\times2}$,
\begin{equation}\label{fcons}
\del \E^\ac(u_\mF) = 0, \quad \text{where} \quad u_\mF(x) : = \mF x.
\end{equation}

$\mathcal{E}^{\rm ac}$ is \emph{patch test consistent} if it satisfies both \eqref{econs} and \eqref{fcons}.

For the sake of brevity of notation we will often write
$V^i_\ell(Du(\ell)) : = V\left((\mathcal{R}_\ell D_ju(\ell))_{j=1}^6\right)$.
Following \cite{PRE-ac.2dcorners} we make the following standing assumption (see
Figure \ref{fig:interface_corners} for examples).

{\bf (A0)} \emph{Each vertex $\ell \in \Is$ has exactly two neighbours in $\Is$, and at least one neighbour in $\Cs$.}

\begin{figure}
\centering
\includegraphics[width=0.9\linewidth]{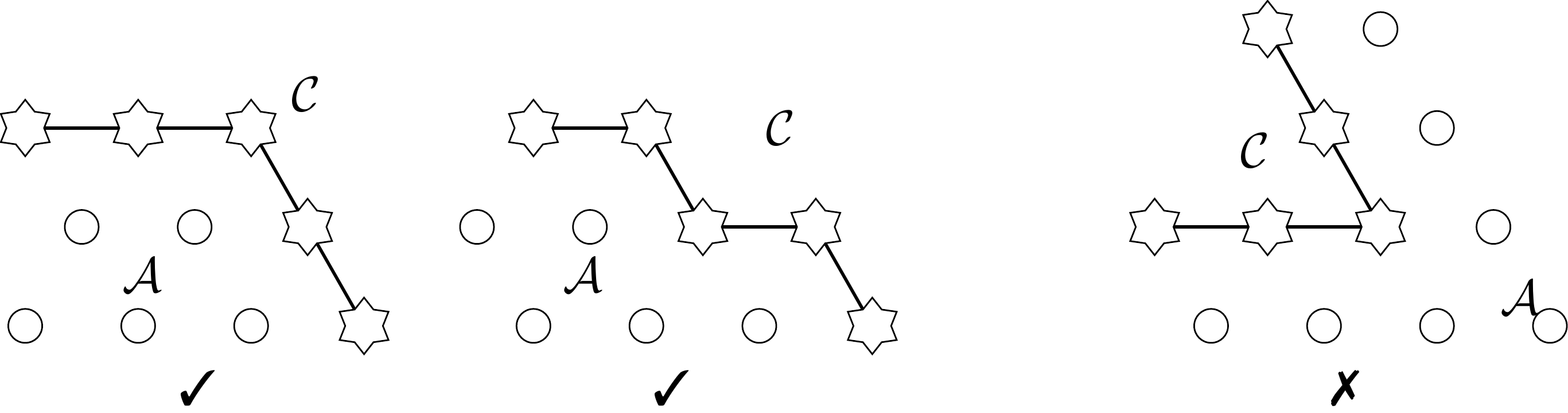}
\caption{\small The first two configurations are allowed. The third configuration is not allowed as the interface atom at the corner has no nearest neighbour in the continuum region, and should instead be taken as an atomistic site.}
\label{fig:interface_corners}
\end{figure}

Under this assumption, the geometry reconstruction operator $\mathcal{R}_\ell$ is then defined by
\begin{equation*}
\begin{aligned}
\mathcal{R}_\ell D_j y(\ell) &:= (1-\lambda_{\ell,j}) D_{j-1}y(\ell) + \lambda_{\ell,j} D_{j}y(\ell) + (1-\lambda_{\ell,j}) D_{j+1}y(\ell),\\
\lambda_{x,j} & :=
\cases{2/3, & x+a_j \in \Cs \\
	1, &\text{otherwise} };
\end{aligned}
\end{equation*}
see Figure \ref{fig:general_interface}. The resulting coupling method is called G23 and the corresponding energy functional $\E^{\rm g23}$. This choice of coefficients (and only this choice) leads to patch test consistency \eqref{econs} and \eqref{fcons}.

\begin{figure}
	\centering
	\includegraphics[width=0.55\linewidth]{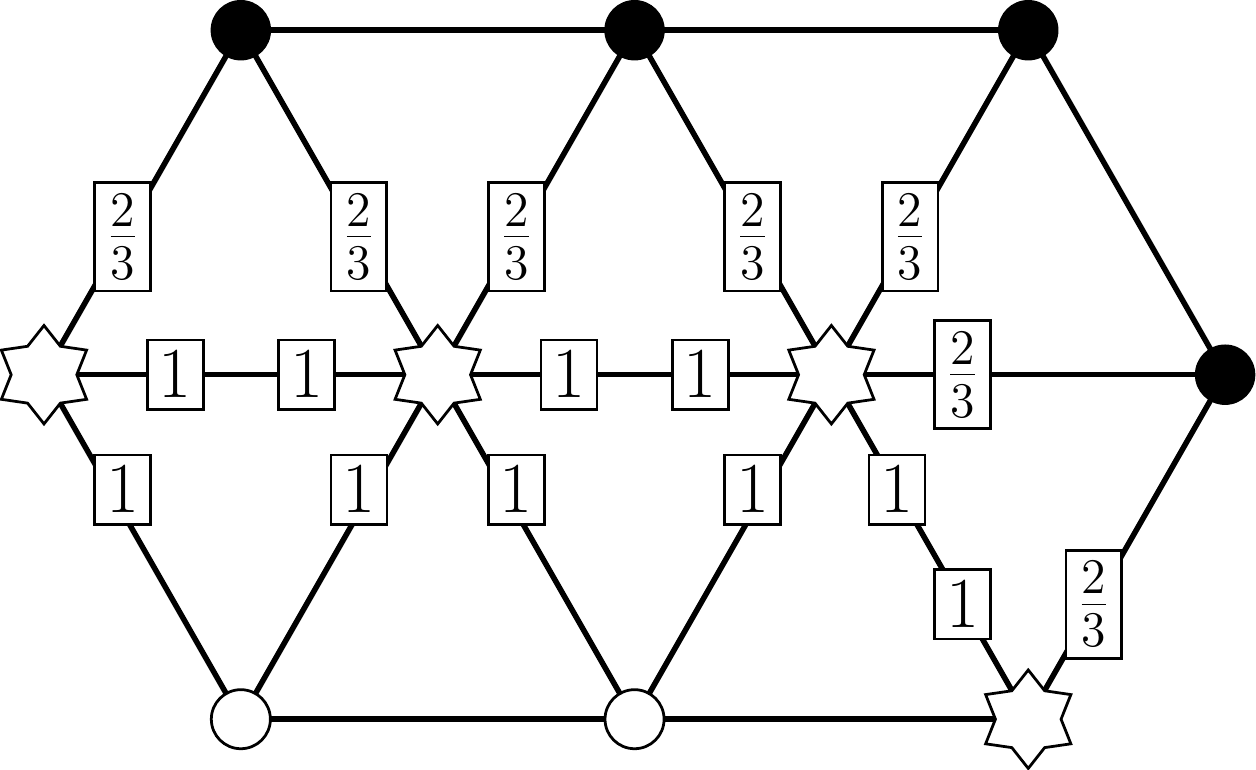}
	\caption{\small The geometry reconstruction coefficents $\lambda_{x,j}$ at the interface sites.}
	\label{fig:general_interface}
\end{figure}

For future reference we decompose the canonical triangulation $\T$ as follows:
\begin{equation}\label{def:ac}
\begin{aligned}
\mathcal{T}_\mathcal{A} :& = \{ T \in\mathcal{T} \,|\, T\cap(\mathcal{I}\cup \mathcal{C}) = \emptyset, \}, \\
\mathcal{T}_\mathcal{C} :& = \{ T \in\mathcal{T}\,|\, T\cap(\mathcal{I}\cup \mathcal{A}) = \emptyset, \} \quad \text{and} \\
\mathcal{T}_\mathcal{I}:& =\mathcal{T}\setminus (\mathcal{T}_\mathcal{C}\cup\mathcal{T}_\mathcal{A}).
\end{aligned}
\end{equation}

\subsection{Notation for a P$2$ finite element scheme}
In the atomistic and interface regions, the interactions are represented by
discrete displacement maps, which are identified with their linear
interpolant. Here, we identify the displacement map with its P1 interpolant. No
approximation error is committed.

On the other hand, in the continuum region where the interactions are
approximated by the Cauchy--Born energy, we could increase the accuracy by using
P$p$-FEM with $p > 1$. In later sections we will review that the Cauchy--Born
approximation yields a 2nd-order error, whereas employing the P1-FEM in the
continuum region would reduce the accuracy to first order. In fact, we will show  in that, with optimized mesh grading, P2-FEM is sufficient to obtain a convergence rate that cannot be improved by other choices of continuum discretisations. High-order P$p$-FEM with $p>2$ will increase the computational costs but yield the same error convergence rate (see \S~\ref{sec:high-order-discussion}).

Let $K>0$ denote the inner radius of the atomistic region,
\begin{equation*}
  K := \sup\b \{ r>0 \,|\, \mathcal{B}_r \cap \L \subset \As    \b \},
\end{equation*}
where $\mathcal{B}_r$ denotes the ball of radius $r$ centred at $0$.  In order
for the defect to be contained in the atomistic region we assume throughout that
$K \geq R_f$.

Let $\Omega_h$ denote the entire computational domain and $N>0$ denote the inner radius of $\Omega_h$, i.e.,
\begin{equation*}
N : = \sup \b \{r>0 \,|\, \mathcal{B}_r  \subset \Omega_h \b \}.
\end{equation*}
Let $\T _h $ be a finite element triangulation of $\Omega_h$ which satisfies, for $T \in \T_h $,
\begin{equation*}
  T \cap (\As \cup \Is) \ne \emptyset \quad \Rightarrow \quad T \in  \T.
\end{equation*}
In other words, $\T _h$ and $\T $ coincide in the atomistic and interface regions, whereas in the continuum region the mesh size may increase
towards the domain boundary. The optimal rate at which the mesh size increases
will be determined in later sections.

We note that the concrete construction of $\T_h$ will be based on the choice of
the domain parameters $K$ and $N$; hence, when emphasizing this dependence, we
will write $\T_h(K,N)$. We assume throughout that the family
$(\T_h(K,N))_{K, N}$ is \emph{uniformly shape-regular}, i.e., there exists $c>0$
such that,
\begin{equation}
  \label{eq:unif-shape-reg}
  \diam  (T) ^2 \le c |T| , \quad \forall T\in \T_h(K, N), \forall K \le N.
\end{equation}
This assumption eliminates the possibility of extreme angles on elements, which
would deteriorate the constants in finite element interpolation error
estimates. For the most part we will again drop the parameters from the notation
by writing $\T_h \equiv \T_h(K, N)$ but implicitly will always keep the
dependence.

Similar to \eqref{def:ac}, we define the atomistic, interface and continuum
elements as $\mathcal{T}_h^a,\mathcal{T}_h^i$ and $\mathcal{T}_h^c$,
respectively. Note that $\mathcal{T}_h^a = \mathcal{T}_\mathcal{A}$ and
$\mathcal{T}_h^i = \mathcal{T}_\mathcal{I}$. We also let $\mathcal{N}_h$ denote
the number of degrees of freedom of $\T _h$.

We define the finite element space of admissible displacements as
\begin{equation}
  \label{eq:defn-Uh}
\begin{aligned}
  \mathcal{U}_h : = \b\{u\in C(\R^2; \R^m) \sep\,&  {\rm{supp}} (u)\subset \Omega_h, u|_T \in \mathbb{P}^1(T) \text{ for } T\subset \mathcal{T}_h^a \cup \mathcal{T}_h^i \,\text{ and} \\
  & u|_T \in \mathbb{P}^2(T) \text{ for } T\subset \mathcal{T}_h^c \b\}.
\end{aligned}
\end{equation}
In defining $\mathcal{U}_h$ we have made two approximations to the class of
admissible displacements: (1) truncation to a finite computational domain and
(2) finite element coarse-graining.

The computational scheme is to find
\begin{equation}\label{eq:u_h}
u^{\rm g23}_h \in \arg \min \b\{ \E^{\rm g23}(u_h) - f(u_h) \sep  u_h \in \mathcal{U}_h\b\}.
\end{equation}

\begin{remark}
  $\mathcal{U}_h$ is embedded in $\mathcal{U}_0$ via point evaluation. Through
  this identification, $f(u_h)$ is well-defined for all $u_h \in
  \mathcal{U}_h$.

  We will make this identification {\em only} when we evaluate $f(u_h)$. The
  reason for this is a conflict when interpreting elements $u_h$ as lattice
  functions is that we identify lattice functions with their continuous
  interpolants with respect to the canonical triangulation $\mathcal{T}$, which
  would be different from the function $u_h$ itself. However, for the evaluation
  of $f(u_h)$ this issue does not arise.
\end{remark}

\section{Summary of results}

\subsection{Regularity of $u^\a$}
The approximation error analysis in later sections requires estimates on the
decay of the elastic fields away from the defect core. These results follow from
a natural stability assumption:

\bf{(A1)} \normalfont The atomistic solution is strongly stable, that is, there
exists $C_0>0$,
\begin{equation}\label{assum:stab_hom}
  \<\del^2 \E^{\a} (u^\a) \varphi, \varphi  \> \ge C_0 \|\D \varphi\|^2_{L^2},
  \quad \forall \varphi \in  \dot{\mathcal{U}}^{1,2},
\end{equation}
where $u^\a$ is a solution to \eqref{eq:y_a}.

\begin{corollary}\label{thm:decay}
  Suppose that {\bf (A1)} is satisfied, then there exists a constant $C>0$
  such that, for $1\le j \le r-2$,
  \begin{equation*}
    |D^ju^\a(\ell)|\le
      C| \ell |^{-1-j}.
  \end{equation*}
\end{corollary}

\begin{proof}
  See Theorem 2.3 in \cite{EhrOrtSha:2013}.
\end{proof}

\subsection{Stability}
In \cite{2013-stab.ac} it is shown that there is a ``universal'' instability in
2D interfaces for QNL-type a/c couplings: it is impossible to prove in full
generality that $\ddel \E^{\rm g23}(u^\a)$ is a positive definite operator, even
if we assume \eqref{assum:stab_hom}. Indeed, this potential instability is
universal to a wide class of generalized geometric reconstruction
methods. However, it is rarely observed in practice. To circumvent this
difficulty, we make the following standing assumption:

\bf{(A2)} \normalfont The {\em homogeneous lattice} is strongly stable under the
G23 approximation, that is, there exists $C^{\rm g23}_0 > 0$ which is independent of $K$
such that, for $K$ sufficiently large,
\begin{equation}\label{assum:stab_hom_g23}
\<\del^2 \E^{\rm g23}(0)\varphi_h, \varphi_h  \> \ge C^{\rm g23}_0\|\D \varphi_h\|^2_{L^2}, \quad \forall \varphi_h \in  \mathcal{U}_h.
\end{equation}

Since \eqref{assum:stab_hom_g23} does not depend on the solution it can be
tested numerically. But a precise understanding under which conditions
\eqref{assum:stab_hom_g23} is satisfied is still missing. In \cite{2013-stab.ac}
a method of stabilizing 2D QNL-type schemes with flat interfaces is introduced,
which could replace this assumption, however we are not yet able to extend this
stabilizing method for interfaces with corners, such as the configurations
discussed in this paper.

\subsection{Main results} \label{sec:main_results}
To state the main results it is convenient to employ a smooth interpolant to
measure the regularity of lattice functions. In Lemma \ref{lem:smooth_int}, we
define such an interpolant $\tilde{u} \in C^{2,1}(\R^2)$ for
$u \in \mathcal{U}_0$, for which there exists a universal constant $\tilde{C}$
such that, for all $q \in [1,\infty]$, $0 \le j \le 3$,
\begin{displaymath}
  |D^j u(\ell)| \leq \tilde{C} \| \nabla^j \tilde{u} \|_{L^1(\omega_\ell)}
  \quad \text{and} \quad
  \|\D^j \tilu \|_{L^q(T)} \le
  \tilde{C} \|D^j u \|_{\ell^q(\Lambda \cap T)}
\end{displaymath}
where $\omega_\ell: = \ell+\mA(-1,1)^2$.

\subsubsection{Consistency error estimate}
In (\ref{def:Pi_h}) we define a quasi-best approximation operator
$\Pi_h : \mathcal{U}_0 \rightarrow \mathcal{U}_h$, which truncates an atomistic
displacement to enforce the homogeneous Dirichlet boundary condition, and then
interpolates it onto the finite element mesh.

Our main result is the following consistency error estimate.

\begin{theorem}\label{thm:consist}
  If $u^\a $ is a solution to \eqref{eq:y_a} then we have, for all
  $\varphi_h \in \mathcal{U}_h$,
	\begin{equation}\label{eq:consist}
	\begin{aligned}
	\langle \del \E^{\rm g23}(\Pi_h u^\a ),\varphi_h\rangle \lesssim & \bigg( \|\D ^2 \tilde{u}^{\a}\|_{L^2(\Omega^{\rm i})}+\|\D ^3 \tilde{u}^{\a} \|_{L^2(\Omega^{\rm c})}
	+ \|\D ^2 \tilde{u}^{\a} \|^2_{L^4(\Omega^{\rm c})}\\
	& + \|h^2\D ^3 \tilde{u}^{\a} \|_{L^2(\Omega_h^{\rm c})} + \|\D  \tilde{u}^{\a} \|_{L^2(\R ^2 \setminus \mathcal{B}_{N/2})} \\
	& + N^{-1} \| h^2\nabla^2 \tilu \|_{L^2(\mathcal{B}_{N} \setminus B_{N/2})} \bigg) \|\D \varphi_h\|_{L^2(\R ^2 \setminus \Omega^\a)},
	\end{aligned}
	\end{equation}
	where $\Omega_h^{\c}$ corresponds to the continuum region of $\Omega_h$, and $h(x): = \text{\rm diam}(T)$ with $x \in T \in \T _h$.
\end{theorem}

\subsubsection{Optimizing the approximation parameters}
\label{sec:optim-params}
Before we estimate the
error $\|\D u^\a -\D u_h\|_{L^2}$, we optimize the approximation parameters in
the computational scheme. This means that the radius $K$ of the atomistic
region, the radius $N$ of the entire computational domain and the mesh size $h$
should satisfy certain balancing relations. We only outline the result of this
optimisation and refer to \S~\ref{sec:optimal_mesh} for the details.

Due to the decay estimates on $\tilde{u}^\a$ the dominating terms in
(\ref{eq:consist}) turn out to be
\begin{equation}
  \label{eq:critical_terms_to_balance}
  \|\D ^2 \tilde{u}^{\a}\|_{L^2(\Omega^{\rm i})}
  \qquad \text{and} \qquad
  \|\D \tilde{u}^{\a} \|_{L^2(\R ^2 \setminus \mathcal{B}_{N/2})}.
\end{equation}
(We will see momentarily that the mesh size plays a minor role.)  These two
terms result from the nature of the coupling scheme and the far-field truncation
error. In particular, both of these cannot be improved by the choice of
discretisation of the Cauchy--Born model, e.g., order of the FEM. We also note
that, if we had employed a P1-FEM, then the limiting factor would have been
$\| h \nabla^2 \tilu^\a \|_{L^2(\Omega^{\rm c})}$.

We can balance the two terms in~\eqref{eq:critical_terms_to_balance} by choosing
$N \approx K^{5/2}$. It then remains to determine a mesh-size so that the finite
element error contribution,
\begin{displaymath}
  \|h^2\D ^3 \tilde{u}^{\a} \|_{L^2(\Omega^{\rm c}_h)} \qquad \text{and} \qquad  N^{-1} \| h^2\nabla^2 \tilu \|_{L^2(\mathcal{B}_{N} \setminus B_{N/2})}
\end{displaymath}
remains small in comparison. We show that the scaling
$h(x) \approx \left( \frac{|x|}{K}\right)^\beta$ is a suitable choice, with
$1 < \beta < 3/2$, under which both terms become of order $O(K^{-3})$.

Thus, we have determined the approximation parameters $(K, N, h)$ in terms
of a single parameter $K$. The quasi-optimal relations for P2-FEM discretisatino
of the Cauchy--Born model are summarised in Table~\ref{tab:params}.

\begin{table}
\begin{center}
  \begin{tabular}{r|c|c|c|c}
      &  $\beta$  & $N$ & $\mathcal{N}_h$ & consistency error \\
    \hline P2-FEM & $ \left(1 , \frac{3}{2} \right) $ & $K^{5/2}$ & $K^2$ & $K^{-5/2} $ \\[1mm]
    \hline P1-FEM & $ \left(1 , \frac{3}{2} \right) $ & $K^{2}$ & $K^2$ & $K^{-2} $ \\[1mm]
  \end{tabular}
\end{center}
\medskip
\caption{Quasi-optimal relations between approximation parameters for P2-GR23
  and, for comparision, for P1-GR23.}
  \label{tab:params}
\end{table}

\begin{corollary} Suppose that $N, h$ satisfy the relations of
  Table~\ref{tab:params}, the consistency error estimate (\ref{eq:consist}) in
  terms of the number of degrees of freedom $\mathcal{N}_h$ can be written as
  \begin{equation}\label{eq:consis_N}
    \|\del \E^{\rm g23}(\Pi_h u^\a )\|_{\mathcal{U}^{-1,2}}
    \lesssim \mathcal{N}_h^{-5/4}.
  \end{equation}
\end{corollary}

\subsubsection{Error estimate} To complete our summary of results, we now use
the Inverse Function Theorem to obtain error estimates for the strains and the
energy.

\begin{theorem}\label{theo:main}
  Suppose that \textbf{(A0), (A1)} and \textbf{ (A2)} are satisfied and that the
  quasi-optimal scaling of $N, h$ from Table \ref{tab:params} is
  satisfied. Then, for sufficiently large atomistic region size $K$, a solution
  $u^{\rm g23}_h$ to \eqref{eq:u_h} exists which satisfies the error estimates
  \begin{align}
  \label{eq:u_error}  \|\D u^\a - \D u^{\rm g23}_h\|_{L^2} &\lesssim \mathcal{N}_h^{-5/4}, \quad \text{and}\\
 \label{eq:E_error}
   \big|[\E^\a (u^\a)-f(u^\a)] - [\E^{\rm g23}(u^{\rm g23}_h) - f(u^{\rm g23})] \big|
   &\lesssim \mathcal{N}_h^{-7/4},
   \end{align}
   where $\mathcal{N}_h$ is the number of degrees of freedom.
\end{theorem}

%
%

\subsection{Setup of the numerical tests}

For our numerical tests, we consider an anti-plane displacement
$u : \L \rightarrow \R$. We choose a hexagonal atomistic region $\Omega^\a$ with
side length $K$ and one layer of atomistic sites outside $\Omega^\a$ as the
interface. To construct the finite element mesh, we add hexagonal layers of
elements such that, for each layer $j$,
$h(\text{layer }j) = (h(\text{layer }j-1)/K)^{\beta}$, with $\beta = 1.4$; see
Figure \ref{fig:coarse_domain}. The procedure is terminated once the radius of
the domain exceeds $N = \lceil K^{5/2} \rceil$. This construction guarantees the
quasi-optimal approximation parameter balance to optimise the P2-FEM error. The
derivation is given in Section \ref{sec:optimal_mesh}.

In our tests we compare the P2-G23 method against
\begin{itemize}
\item[(1)] a pure atomistic model with clamped boundary
condition: the construction of the domain is as in the P2-G23 method, but
without continuum region;
\item[(2)] a P1-G23 method: the construction is again identical to that of the
  P2-G23 method, but the P2-FEM in the definition of $\mathcal{U}_h$ is replaced
  by a P1-FEM. The same mesh scaling as for P2 is used (see also \cite{E:2006} where this is shown to be quasi-optimal).
\end{itemize}

The site potential is given by a nearest-neighbour embedded atom toy model,
\begin{equation*}
  V (Du): = G\left( \sum_{i = 1}^6 \rho(D_i u(\ell))\right)
\end{equation*}
with $G(s) := s+\frac12 s^2$ and $\rho(r) := \sin^2 (r\pi)$. This is the
anti-plane toy model as the one used in \cite{EhrOrtSha:2013}.

The external potential is defined by $\<f, u\> = 10 (u(0,0) - u(1,0))$, which
can be thought of as an elastic di-pole.  A steepest descent method,
preconditioned with a finite element Laplacian and fixed (manually tuned)
step-size, is used to find a minimizer $u^{\rm g23}_h $ of
$ \E^{g23}(u) - f(u)$, using $u_h = 0$ as the starting guess.

In order to compare the errors, we use a comparison solution with atomistic
region size $3K$ and other computational parameters scaled as above.

The numerical results, with brief discussions, are shown in Figures
\ref{fig:screw3_err2_k2}--\ref{fig:screw3_errE_dofs}. The two most important
observations are the following:
\begin{itemize}
\item[(1)] the numerical tests confirm the analytical predictions for the
  energy-norm error, but the experimental rates for the energy error are better
  than the analytical rates. Similar observations were also made in
  \cite{EhrOrtSha:2013}.
\item[(2)] With our specific setup, the improvement of the P2-GR23 over P1-GR23
  is clearly observed when plotting the error against
  $\#\mathcal{A} \propto \mathcal{N}_h$, but when plotted against
  $\mathcal{N}_h$ the improvement is only seen in the asymptotic regime. This
  indicates that further work is required, such as a posteriori adaption, to
  optimise the P2-GR23 in the pre-asymptotic regime as well.
\end{itemize}

\begin{figure}
	\centering
	\includegraphics[width=0.5\linewidth]{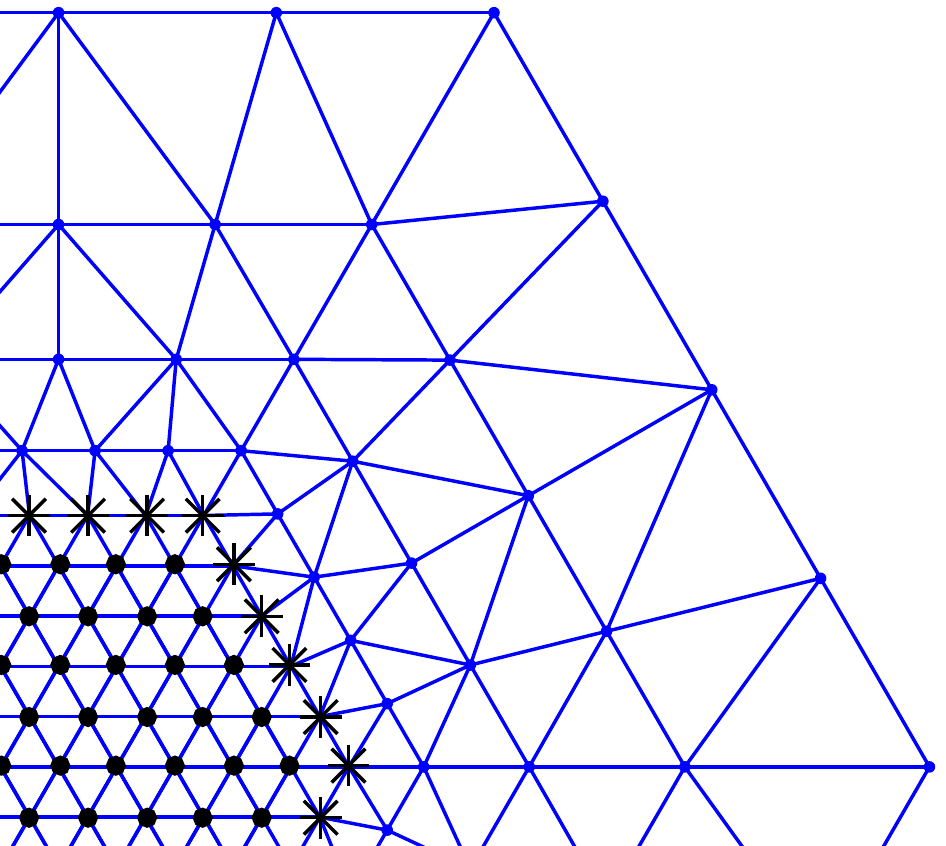}
	\caption{\small An example of the computatioanl mesh. The the vertices marked by "\textbullet" are the atomistic sites; the vertices marked by "$ \ast$" are the interface sites.}
	\label{fig:coarse_domain}
\end{figure}

\begin{figure}
	\centering
	\includegraphics[width=0.8\linewidth]{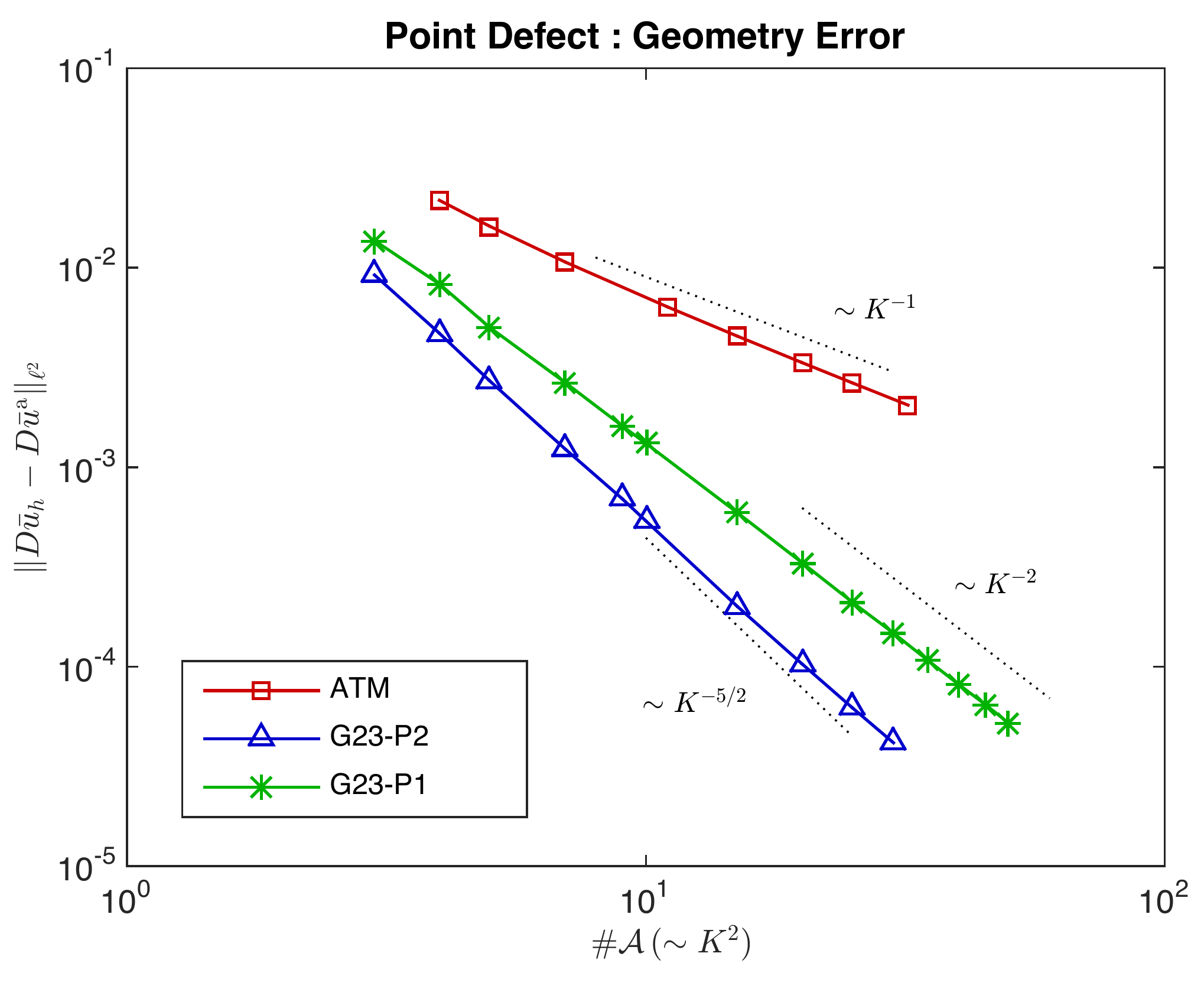}
	\caption{Error in energy norm plotted against $\# \As $. We clearly observe the
          predicted rate of convergence.}
	\label{fig:screw3_err2_k2}
\end{figure}

\begin{figure}
	\centering
	\includegraphics[width=0.8\linewidth]{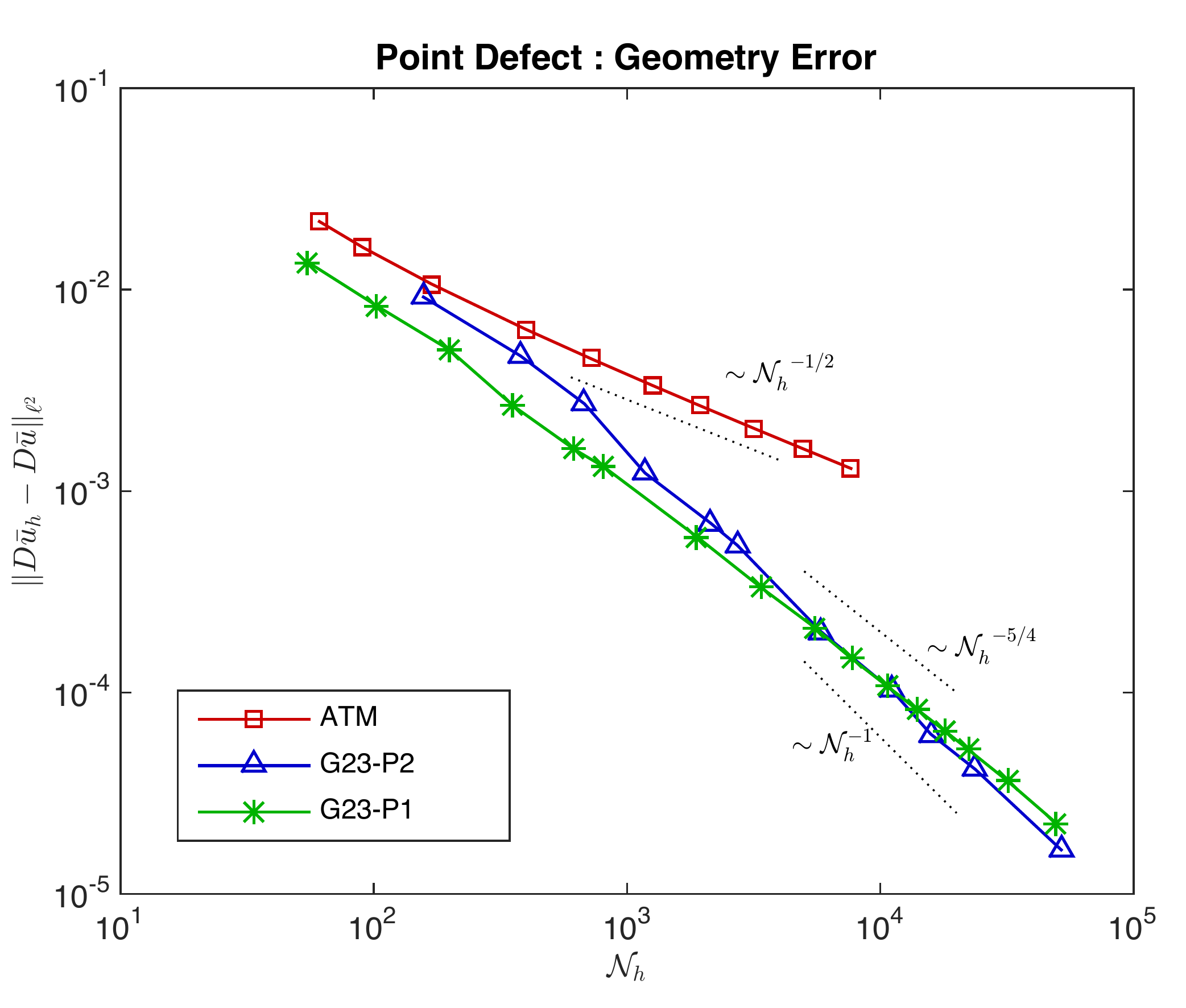}
	\caption{\small{Error in energy norm plotted against the number of
            degrees of freedom. The improvement of P2-FEM is now only seen
            asymptotically.}}
	\label{fig:screw3_err2_dofs}
\end{figure}

\begin{figure}
	\centering
	\includegraphics[width=0.8\linewidth]{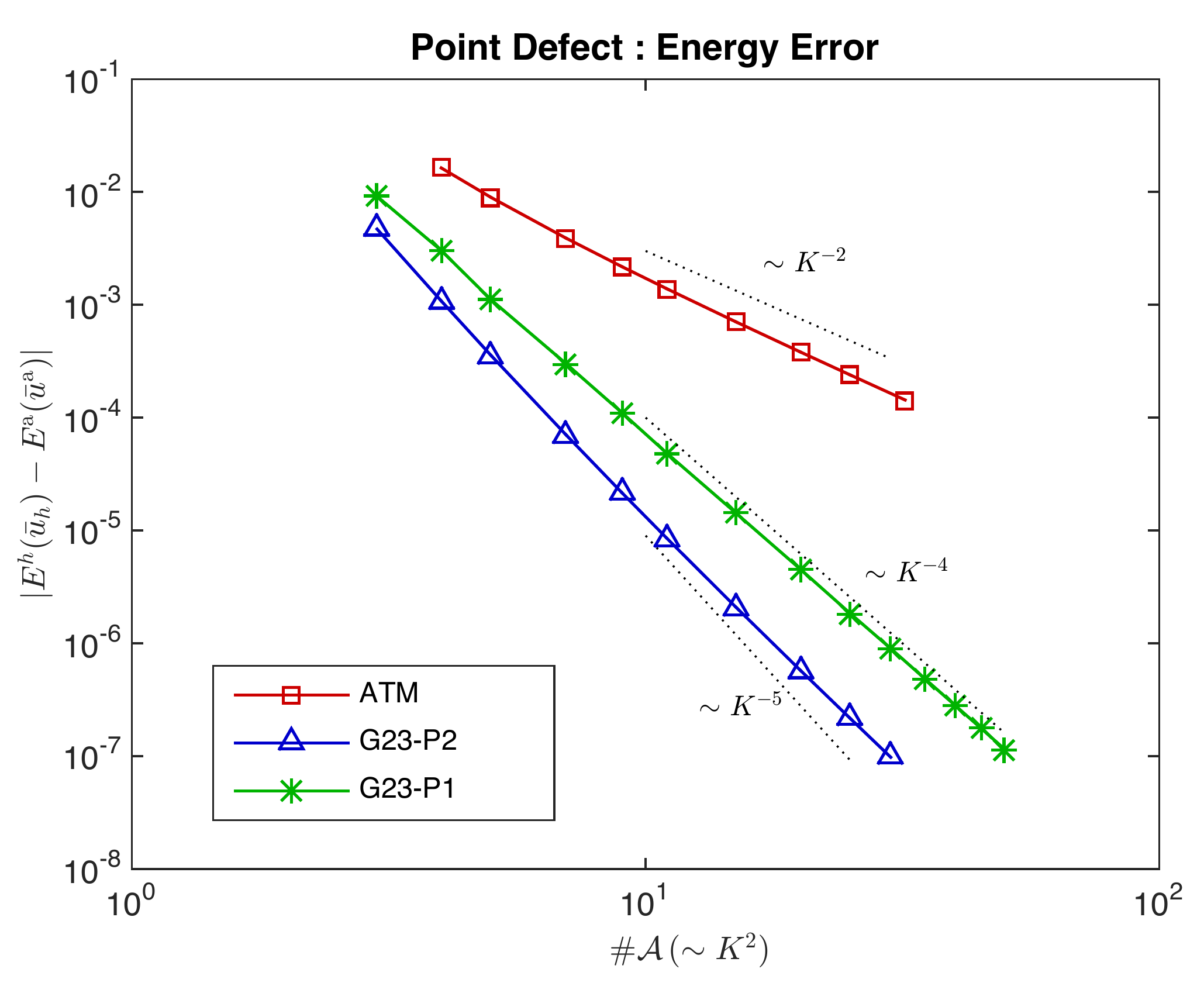}
	\caption{\small{The energy error plotted against $\# \As$. The observed
            rate of convergence is better than the rate predicted in Theorem
            \ref{theo:main}.}}
	\label{fig:screw3_errE_k2}
\end{figure}

\begin{figure}
	\centering
	\includegraphics[width=0.8\linewidth]{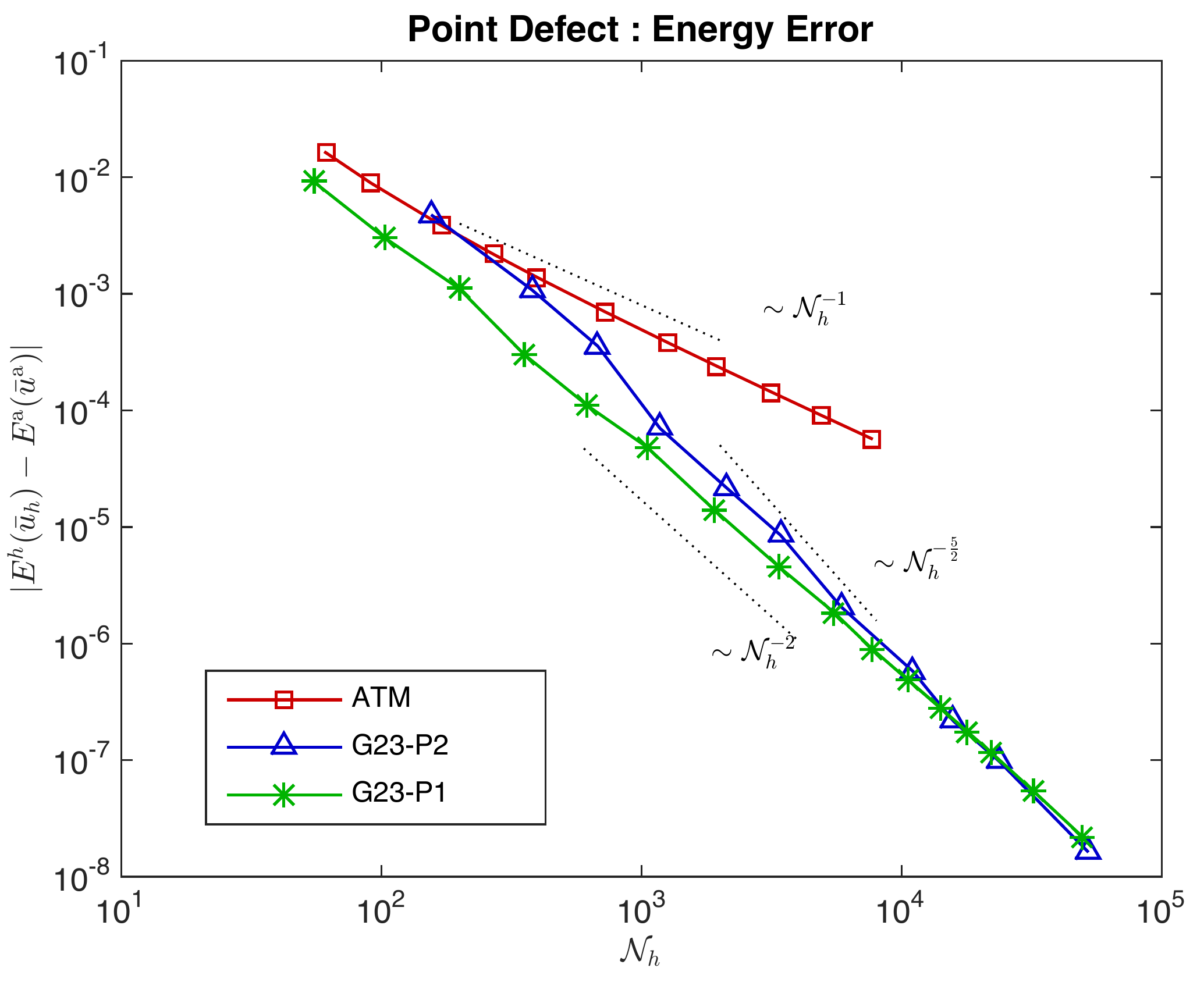}
	\caption{\small{The energy error plotted against the number of degrees
            of freedom. The improvement of P2-FEM over P1-FEM can again only be
            observed asymptotically.}}
	\label{fig:screw3_errE_dofs}
\end{figure}

\subsection{Extension to high-order FEM}
\label{sec:high-order-discussion}
If we apply higher-order FEM in the continuum region, then to extend our error
analysis we would need a smooth interpolant of $u \in \mathcal{U}_0$ with higher
regularity than $\tilde{u} \in C^{2,1}(\R ^2)$. A suitable extension given in \cite{2014-bqce} is, for arbitrary $n$, a $C^{n,1}$ piecewise polynomial
of degree $2n+1$ with properties analogous to those stated in Lemma
\ref{lem:smooth_int}. The resulting higher-order decay rate $|\D ^j \tilde{u}^\a
(x)| \lesssim |x|^{-j-1}$ indicates that the use of high-order FEM could be
beneficial.

However, as we have pointed out in \S~\ref{sec:optim-params}, if we employ the
mesh grading $h(x) =  (|x|/K)^\beta$ with $1<\beta<3/2$ in the continuum
region, the total approximation error cannot be improved by using P$p$-FEM
with $p>2$, since the dominating term is the interface error $\|\D
^2\tilde{u}^\a \|_{L^2(\Omega ^{\rm i})}$ for $p \ge 2$, which results from
the construction of G23 coupling and is not affected by the choice of FEM.

If we consider a coarser mesh for high-order FEM in hopes of reducing
the number of degrees of freedom, i.e., choosing $\beta \ge 3/2$, then applying
analogous calculations to those in \S \ref{sec:optimal_mesh} gives us the following
result:

{\it Employing P$p$-FEM with $p>2$, in order to match the convergence rate of the
   Cauchy--Born error term $\|\D ^3 \tilde{u}^\a \|_{L^2(\Omega^ \c )}\sim
   K^{-3}$, the highest mesh coarsening rate is
\begin{equation*}
   \beta = \frac{5}{3} - \frac{1}{3p}.
\end{equation*}
}

This means that the optimal mesh grading that P$p$-FEM can achieve without
compromising accuracy is no greater than $\frac{5}{3}$. However, in that case,
the number of degrees of freedom is always $\mathcal{O}(K^2)$.

In summary, despite the possibility of (slightly) reducing the number of
degrees of freedom, considering its computational complexity, we conclude
that higher-order FEM is not worthwhile to pursue. However, we emphasize
that this conclusion would need to be revisited if a coupling method with
higher-order interface error as well as continuum model error could be
devised.

\section{Conclusion}

We obtained a sharp energy-norm error estimate for the G23 coupling method with
P2-FEM discretisation of the continuum model. Furthermore, we demonstrated that,
with P1-FEM discretisation the FEM coarsening error is the dominating term in
the consistency error estimate, whereas for P2-FEM discretisation the interface
error becomes the dominating term. In particular, a P2-FEM discretisation yields
a more rapid decay of the error. Crucially though, since for P$p$-FEM with
$p \geq 2$ the interface contribution dominates the total error the P2-FEM is
already optimal. That is, increasing to $p > 2$ will not improve the rate of
convergence, but increase the computational cost and algorithmic complexity.

Numerically, we observe that the improvement of P2-GR23 over P1-GR23 is only
modest at low $\mathcal{N}_h$, hence a P2-GR23 scheme would be primarily of
interest if high accuracy of the solution is required. However, our numerical
results indicate that there is scope for further optimisation, using a
posteriori type techniques.

While our estimates for the error in energy-norm are sharp, our numerical
results show the estimates for the energy errors are suboptimal. We hightlight
the leading term in the error analysis which overestimate the error in Section
\ref{sec:energy_error}. We are unable, at present, to obtain an optimal energy
error estimate. This appears to be an open problem throughout the literature on
hybrid atomistic multi-scale schemes; see e.g. \cite{EhrOrtSha:2013}.

In summary we conclude that using P2-FEM is a promising improvement to the
efficiency of a/c coupling methods, but that some further work, both theoretical
and for its implementation may be need to exploit its full potential.

\section{Reduction to consistency}
\label{reduction to consistency}

Assuming the existence of an atomistic solution $u^\a$, we seek to prove the
existence of $u_h^{\rm g23} \in \mathcal{U}_h$ satisfying
\begin{equation}\label{dE_h}
  \< \del \E^{\rm g23} (u^{\rm g23}_h), \varphi_h\>
  = \<\delta f(u_h^{\rm g23}), \varphi_h\>,
  \quad \text{for all }\varphi_h \in \mathcal{U}_h,
\end{equation}
and to estimate the $\|u^\a - u^{\rm g23}_h\|$ in a suitable norm.

The error analysis consists of consistency and stability
estimates. Once these are established we apply the following theorem to obtain
the existence of a solution $u_h^{\rm g23}$ and the error estimate. The proof of
this theorem is standard and can be found in various references,
e.g. \cite[Lemma 2.2]{Ortner:qnl.1d}.

\begin{theorem}[The inverse function theorem]\label{theo:inverse}
  Let $\mathcal{U}_h$ be a subspace of $\mathcal{U}$, equipped with
  $\|\nabla \cdot \|_{L^2}$, and let
  $\mathcal{G}_h \in C^1(\mathcal{U}_h,\mathcal{U}_h^*)$ with
  Lipschitz-continuous derivative $\delta \mathcal{G}_h$:
\begin{equation*}
\|\delta \mathcal{G}_h(u_h)-\delta \mathcal{G}_h(v_h)\|_\mathcal{L} \le M \|\nabla u_h-\nabla v_h\|_{L^2} \quad \text{for all } u_h,v_h \in \mathcal{U}_h,
\end{equation*}
where $\|\cdot \|_\mathcal{L}$ denotes the
$\mathcal{L}(\mathcal{U}_h,\mathcal{U}_h^*)$-operator norm.

Let $\bar{u}_h \in \mathcal{U}_h$ satisfy
\begin{align}
\|\mathcal{G}_h(\bar{u}_h)\|_{\mathcal{U}_h ^*} &\le \eta, \label{inv_consist} \\
 \langle \delta \mathcal{G}_h(\bar{u}_h)v_h,v_h\rangle &\ge \gamma \|\nabla v_h\|^2_{L^2} \quad
 \text{ for all }v_h \in \mathcal{U}_h, \label{inv_stab}
\end{align}
such that $M,\eta, \gamma$ satisfy the relation
\begin{equation*}
\frac{2M\eta}{\gamma^2}<1.
\end{equation*}
Then there exists a (locally unique) $u_h \in \mathcal{U}_h$ such that $\mathcal{G}_h(u_h) = 0$,
\begin{align*}
\|\nabla u_h-\nabla \bar{u}_h\|_{L^2}&\le 2\frac{\eta}{\gamma}, \quad \text{and} \\
 \langle \delta \mathcal{G}_h(u_h)v_h,v_h\rangle & \ge \left( 1- \frac{2M\eta}{\gamma^2}\right)\gamma \|\nabla v_h\|^2_{L^2} \quad \text{for all } v_h \in \mathcal{U}_h.
\end{align*}
\end{theorem}

To ensure Dirichlet boundary conditions, we adapt the approximation map defined in \cite{EhrOrtSha:2013}. Let $\mu\in C^1(\mathbb{R}^2)$ be a cut-off function such that
\begin{equation*}
\mu(x) = \left\{
\begin{array}{l l}
1 & 0 \le x \le \frac{1}{2},\\
0 & x \ge 1.
\end{array}
\right.
\end{equation*}
For $u : \L \rightarrow \R ^{m}$, define
\begin{equation}\label{def:cut-off}
\mathcal{L} u(x): = \mu\left(\frac{|x|}{N}\right)\left(\tilde{u}(x)-a_u  \right), \text{ where } a_u:= \frac{1}{|B_N \setminus B_{N/2}|} \int_{B_N \setminus B_{N/2}}\tilde{u}(y) \dy.
\end{equation}

Let $\nu_{T,i}, i = 1, 2, 3$ be the vertices of $T$ and $m_e$ be the mid-point
of an edge $e$. Then, the set of all {\em active} P2 finite element nodes
is given by
\begin{equation*}
  \mathcal{N}_h : = \{\nu_{T,i} \sep T \in \T_h, i = 1,2,3\} \cup \{ m_e \sep e  = T_1 \cap T_2, T_1, T_2 \in \T ^c_h\}.
\end{equation*}
This includes all P1 nodes as well as the P2 nodes (edge midpoints) associatd
with edges entirely in the P2 region.

Furthermore, let $I^2_h: C(\R^2; \R^m) \rightarrow \mathcal{U}_h$ be the
interpolation operator such that, for $g \in C(\R^2; \R^m)$,
$I^2_h(g)|_T \in \mathbb{P}^1(T)$ for
$ T\subset \mathcal{T}_h^a \cup \mathcal{T}_h^i $,
$ I^2_h(g)|_T \in \mathbb{P}^2(T) \text{ for } T\subset \mathcal{T}_h^c$,
and
\begin{displaymath}
I^2_h(g)(x) =  g(x) \qquad \text{ for all $x \in \mathcal{N}_h$}.
\end{displaymath}

\begin{remark}
  We also introduce ghost nodes on the edges shared by interface and continuum
  elements:
  \begin{equation}
    \mathcal{N}_h^g : = \{ m_e \sep e = T_1 \cap T_2, T_1 \in \T ^i_h , T_2 \in \T^c_h  \}.
  \end{equation}
  Then, for $x \in \mathcal{N}_h^g$,
  $I^2_h(g)(x) = (g(\nu_x^1) + g(\nu_x^2))/2$, where $\nu_x^1$ and $\nu_x^2$ are
  the vertices of the edge on which $x$ lies. Hence, the $P^1$ and $P^2$
  interpolants coincide on $\mathcal{N}_h^g$.
\end{remark}

We can now define the projection map (quasi-best approximation operator)
$\Pi_h:\mathcal{U}_0 \rightarrow \mathcal{U}_h$ as
\begin{equation}\label{def:Pi_h}
\Pi_h: = I^2_h \circ \mathcal{L}.
\end{equation}

\subsection{Stability}
To put Theorem \ref{theo:inverse} (Inverse Function Theorem) into our context,
let
\begin{equation*}
\mathcal{G}_h(v) : =  \del \E^{\rm g23}(v) -\del f(v) \quad \text{and}\quad \bar{u}_h := \Pi_h u^\a.
\end{equation*}
To make \eqref{inv_consist} and \eqref{inv_stab} concrete we will show that
there exist $\eta,\gamma >0$ such that, for all $\varphi_h \in \mathcal{U}_h$,
\begin{equation*}
\begin{aligned}
\langle \del \E^{\rm g23}(\Pi_h u^\a ),\varphi_h\rangle - \<\del f(\Pi_h u^\a), \varphi_h\>&\le \eta \|\D \varphi_h\|_{L^2}, \quad (consistency) \\
\langle \del^2 \E^{\rm g23}(\Pi_h u^\a ) \varphi_h, \varphi_h \rangle -\<\del^2 f(\Pi_h u^\a)\varphi_h, \varphi_h\> &\ge \gamma  \|\D \varphi_h\|^2_{L^2}. \quad (stability)
\end{aligned}
\end{equation*}
Ignoring some technical requirements, the inverse function theorem implies that, if $\eta / \gamma $ is sufficiently small, then there exists $u^{\rm g23}_h \in \mathcal{U}_h$ such that
\begin{equation*}
\begin{aligned}
&\langle \del \E^{\rm g23}(u_h^{\rm g23} ),\varphi_h\rangle - \< \del f(u_h^{\rm g23}), \varphi_h\>=0, \quad \forall \varphi_h \in \mathcal{U}_h, \quad \text{and}\\
&\|\D u^{\rm g23}_h - \D \Pi_h u^\a\|_{L^2}  \le 2 \frac{\eta}{\gamma}.
\end{aligned}
\end{equation*}
Finally adding the best approximation error $\|\D \Pi_h u^\a - \D u^\a\|_{L^2}$
gives the error estimate
\begin{displaymath}
  \|\D u^{\rm g23}_h - \D u^\a\|_{L^2} \leq
  \|\D \Pi_h u^\a - \D u^\a\|_{L^2} + 2 \frac{\eta}{\gamma}
\end{displaymath}

The Lipschitz and consistency estimates require assumptions on the boundedness
of partial derivatives of $V$. For $\bfg\in \R^{m\times 6}$, define the first
and second partial derivatives, for  $i,j = 1,\dots,6$, by
\begin{equation*}
\pp_jV(\bfg) := \frac{\pp V(\bfg)}{\pp g_j} \in \R^m, \quad \text{and} \quad \pp_{i,j}V(\bfg) : = \frac{\pp^2 V(\bfg)}{\pp g_i\pp g_j} \in \R^{m\times m},
\end{equation*}
and similarly for the third derivatives
$\pp_{i,j,k}V(\bfg)\in \R^{m\times m\times m}$. We assume that the second and
third derivatives are bounded
\begin{align}
M_2:& = \sum_{i,j=1}^6 \sup_{\bfg\in \R^{m\times 6}} \sup _{\substack{h_1,h_2 \in \R^2,\\|h_1| = |h_2| = 1}} \pp_{i,j} V(\bfg)[h_1,h_2] < \infty, \quad \text{and} \label{def:M_2}\\
M_3:& = \sum_{i,j,k=1}^6 \sup_{\bfg\in \R^{m\times 6}} \sup _{\substack{h_1,h_2,h_3 \in \R^2,\\ |h_1| = |h_2|=|h_3 |= 1}} \pp_{i,j,k} V(\bfg)[h_1,h_2,h_3] < \infty. \label{def:M_3}
\end{align}

With the above bounds it is easy to show that
\begin{equation}\label{eq:lip_V}
\sum_{i=1}^6|\pp_i V(\bfg) - \pp_i V(\bfh)| \le M_2 \max_{j = 1,\dotsc,6} |g_j-h_j|, \quad \text{for }\bfg,\bfh \in \R^{m\times 6}.
\end{equation}

From the bounds above we can obtain the following Lipschitz continuity and
stability results.

\begin{lemma}
  \label{th:Lip}
  There exists $M>0$ such that
  \begin{equation}\label{eq:lip}
    \|\delta \mathcal{G}_h(u_h)-\delta \mathcal{G}_h(v_h)\|_\mathcal{L} \le M \|\nabla u_h-\nabla v_h\|_{L^2} \quad \text{for all } u_h,v_h \in \mathcal{U}_h.
  \end{equation}
\end{lemma}
\begin{proof}
  The result follows directly from the global bounds of derivatives of $V$ and
  the fact that $f \in C^k(\dot{\mathcal{U}}^{1,2})$ and that $\del f$ is
  compactly supported hence $\del^2f$ is also Lipschitz.
\end{proof}

\begin{lemma}\label{lem:stab}
  Under the assumptions \textbf{(A1)} and \textbf{(A2)}, if
  $\mathcal{G}_h(v) : = \del \E^{\rm g23}(v) -\del f(v)$, then there exits $\gamma>0$ such that, when $K$ is sufficiently large,
  \begin{equation}\label{eq:stab}
    \langle \delta \mathcal{G}_h(\Pi_h u^\a)\varphi_h,\varphi_h\rangle \ge \gamma \|\nabla \varphi_h\|^2_{L^2} \quad
    \text{ for all }\varphi_h \in \mathcal{U}_h.
  \end{equation}
\end{lemma}
\begin{proof}
  The proof of this result is a straightforward adaption of the proof of
  \cite[Lemma 4.9]{2014-bqce}, which is an analogous result for blending-type
  a/c coupling.
\end{proof}

\section{Consistency estimate with a P2-FEM}\label{sect:consist}

\subsection{Outline of the consistency estimate}
We begin by decomposing the  consistency error into
\begin{align}
  \notag
\langle \del \E^{\rm g23}(\Pi_h u^\a ),\varphi_h\rangle - \<\del f(\Pi_h u^\a), \varphi_h\> &= \left\{\langle \del \E^{\rm g23}(\Pi_h u^\a ),\varphi_h\rangle - \<\del \E^{\a}( u^\a ),\varphi\rangle \right\} \\
\notag
& \quad + \left\{ \< \del f(\Pi_h u^\a ), \varphi_h\> - \<\del f(u^\a), \varphi\> \right\} \\
\label{eq:eta-int+eta-ext}
& =: \eta_{\rm int} + \eta_{\rm ext},
\end{align}
where $\varphi_h \in \mathcal{U}_h$ is given and we can choose
$\varphi \in \mathcal{U}_0$ arbitrarily.

For $\varphi_h \in \mathcal{U}_h$, $\varphi_h|_T \in \mathbb{P}^2(T)$ for
$T \in \mathcal{T}_h^c$. But the test function $\varphi$ in
$\langle \delta\mathcal{E}^\a (u^\a), \varphi\rangle$ is a piecewise linear
lattice function. While we postpone the construction of $\varphi$, we will
ensure that it is defined in such a way that $\varphi(\ell) = \varphi_h(\ell)$
for all $\ell \in \As \cup \Is \cup \Is ^+$, where $\Is^+$ is an extra layer of
atomistic sites outside $\Is$. With this assumption in place, we can further
decompose $\eta_{\rm int}$ into the following parts,
\begin{equation}\label{eq:decomp}
\begin{aligned}
\eta_{\rm int}
& =\int_{\Omega^\c} \partial_\mathsf{F} W(\nabla \tilde{u}^\a):(\nabla \varphi_h-\nabla \varphi)\\
& \quad +\int_{\Omega^\c} (\partial_\mF W(\D \Pi_h u^\a) - \partial_\mathsf{F} W(\nabla \tilde{u}^\a)):\nabla \varphi_h\\
& \quad
 + \int_{\Omega^\c} \big[\partial_{\sf F} W(\nabla \tilde{u}^\a) - \partial_{\sf F} W(\nabla {u}^\a)\big]: \nabla \varphi \\[1mm]
& \quad +
\< \delta \E ^{\rm g23}(u^\a) -  \delta \E ^\a (u^\a), \varphi \> \\
&=: \delta_1+\delta_2+\delta_3+\delta_4,
\end{aligned}
\end{equation}
where $\tilde{u}^\a$ is the smooth interpolant of $u^\a$ defined in Lemma
\ref{lem:smooth_int} below.  By $\nabla \varphi$ in $\delta_1$ we mean the gradient of
the canonical linear interpolant of $\varphi$. To estimate $\delta_2$ we require
an approximation error estimate for $\Pi_h u - u$. To estimate
  $\delta_3$ we will exploit the fact that the atomistic triangulation
  $\mathcal{T}$ is uniform to prove a super-convergence estimate.  Finally, for the
modelling error, $\delta_4$, we employ the techniques developed in
\cite{PRE-ac.2dcorners}.

To define the smooth interpolant $\tilde{u}^\a $, we use the construction from \cite{2014-bqce}, namely a
$C^{2,1}$-conforming multi-quintic interpolant.
Although the interpolant defined
in \cite{2014-bqce} is for lattice functions on $\Z ^2$, we can use the linear
transformation from $\Z ^2$ to $\L = \mA \Z ^2$ to obtain a modified
interpolant.

\begin{lemma}\label{lem:smooth_int}
	(a) For each $u :\L \rightarrow \R ^m$, there exists a unique $\tilu \in C^{2,1}(\R^2; \R^m)$ such that, for all $\ell \in \L$,
	\begin{equation*}
	\begin{aligned}
	\left.\tilu\right|_{\ell + \mA(0,1)^2} &\text{ is a polynomial of degree 5}, \\
	\tilu (\ell) &= u(\ell),\\
	\partial_{a_i} \tilu (\ell) &= \tfrac{1}{2}\left(u(\ell + a_i) - u(\ell - a_i)\right),\\
	\partial_{a_i}^2 \tilu (\ell)& = u(\ell+a_i) - 2u(\ell) +u(\ell- a_i),
	\end{aligned}
	\end{equation*}
	where $i \in \{1,2\}$ and $\partial_{a_i}$ is the derivative in the
        direction of $a_i$.

	(b) Moreover, for $q \in [1,\infty]$, $0\le j \le 3$,
	\begin{equation}\label{eq: int_ineq_proof}
          \|\D^j \tilu \|_{L^q(\ell+\mA(1,0)^2)} \lesssim \|D^j u \|_{\ell^q\left(\ell + \mA\{-1,0,1,2\}^2 \right)} \quad \text{and } \quad  |D^j u(\ell)| \lesssim\|\D^j \tilu \|_{L^1(\ell+\mA(-1,1)^2)},
	\end{equation}
	where $D$ is the difference operator defined in \eqref{def:diff_op}.  In
        particular,
	\begin{equation*}\|\D \tilde{u}\|_{L^q} \lesssim \|\D u\|_{L^q} \lesssim \|\D \tilde{u}\|_{L^q},
	\end{equation*}
	where $ u $ is identified with its piecewise affine interpolant.
\end{lemma}
\begin{proof} Let $v:\Z^2 \rightarrow \R ^m$ and $v(\xi) := u(\mA \xi)$ for all
  $\xi \in \Z$. Then \cite[Lemma 1]{2014-bqce} shows that there exists a unique
  $\tilv \in C^{2,1}(\R^2; \R^m)$ such that, for $\xi \in \Z^2$,
  \begin{equation*}
   \begin{aligned}
      \left.\tilv\right|_{\xi + (0,1)^2} &\text{ is a polynomial of degree 5}, \\
      \tilv (\xi) &= v(\xi),\\
      \partial_{e_i} \tilde{v}(\xi) &= \tfrac{1}{2}\left(v(\xi + e_i) - v(\xi - e_i)\right), \\
      \partial_{e_i}^2 \tilde{v}(\xi) &= v(\xi+e_i) - 2v(\xi) +v(\xi- e_i) \quad i = 1, 2,
    \end{aligned}
  \end{equation*}
  Defining $\tilu(x) := \tilv(\mA^{-1}x)$ for all $x\in \R^2$ proves part(a).

  For part (b), \cite[Lemma 1]{2014-bqce} establishes also that there exists a
  constant $C'_{j}$ such that, for $\xi \in \Z^2$, $1\le j \le3$, $q\in [1,\infty]$,
\begin{equation*}
\|\D^j \tilv \|_{L^q(\xi+(1,0)^2)} \le C'_{j}\|\hat{D}^j v \|_{\ell^q\left(\xi + \{-1,0,1,2\}^2 \right)},
\end{equation*}
where $\hat{D}$ represents the 4-stencil difference operator in $\Z^2$: let $\mathcal{R} := \{ \rho \in \Z^2 \,|\, |\rho| = 1\}$, then $\hat{D}v(\ell): = (\hat{D}_\rho v(\ell))_{\rho\in \mathcal{R}}$ with $\hat{D}_\rho v(\ell):= v(\ell+\rho)-v(\ell)$. After transformation, we have, for $\xi = \mA \ell \in \L$,
\begin{equation*}
\hat{D}v(\ell) = (D_i u(\xi))_{i = 1,2,4,5}.
\end{equation*}
By adding the additional stencil elements $D_3, D_6$ we obtain
\begin{equation*}
\begin{aligned}
C''_{j}\|\D^j \tilu \|_{L^q(\xi+\mA(1,0)^2)} &\le \|\D^j \tilv \|_{L^q(\ell+(1,0)^2)} \\
&\le C'_{{j}}\|\hat{D}^j v \|_{\ell^q\left(\ell + \{-1,0,1,2\}^2 \right)}\le C'''_{j}\|D^j u \|_{\ell^q\left(\xi + \mA\{-1,0,1,2\}^2 \right)},
\end{aligned}
\end{equation*}
where $C''_{j}$ and $C'''_{j}$ only depend on $j$. Writing
$C: = \max_{1\le j\le 3}\left(\frac{C'''_{j}}{C''_{j}}\right)$ yields the first
inequality of \eqref{eq: int_ineq_proof}. Following a similar argument the second inequality also holds.
\end{proof}

\subsection{Construction of $\varphi$ and estimation of $\delta_1$}

\label{sec:estimate-delta1}
Recall that
\begin{equation*}
\delta_1 := \int_{\Omega^\c} \partial_\mathsf{F} W(\nabla \tilde{u}^\a):(\nabla \varphi_h-\nabla \varphi).
\end{equation*}

We adapt the modified quasi-interpolation operator introduced in
\cite{carstensen} to approximate a test function $\varphi_h \in \mathcal{U}_h$.
The advantage of this interpolation operator is that by using the setting of a
partition of unity the approximation error has a local average
zero. Consequently we can apply Poincar\'{e} inequality on patches to obtain
local estimates.

We think of the construction of $\varphi$ as a Dirichlet boundary problem with
the outer boundary $\partial \Omega_h$ and the inner boundary
$\partial \Omega^\c$.  Let $\phi_\ell$ be the piecewise linear hat-functions on
the canonical triangulation $\T $ associated with $\ell \in \L $. Define
\begin{equation*}
\phi^{\rm PU}_\ell := \frac{\phi_\ell}{\sum_{k\in \Cs \cap \Omega_h} \phi_k } , \quad \forall \ell \in \Cs,
\end{equation*}
where $\Cs$ is the continuum lattice sites as defined in \S \ref{sec:g23model}.
It is clear that $\{ \phi^{\rm PU}_\ell\}_{\ell \in \Cs \cap \Omega_h}$ is a
partition of unity.

Now we refer to \cite{carstensen} for the contruction of a linear interpolant of $\varphi_h \in \mathcal{U}_h$ . We shall define the interpolant as follows:
\begin{equation}\label{def: varphi_intp}
\Pi_h^* \varphi_h(x) :=  \varphi(x) := \varphi_1(x)+\varphi_2(x), \quad \forall x \in \R ^2,
\end{equation}
where 
\begin{equation*}
\begin{aligned}
\varphi_1(\ell) &:= \cases{
\varphi_h(\ell),  & \ell \in \As \cup \Is \cup \Is^+, \\
\frac{\int_{\R ^2}\phi_\ell \varphi_h }{\int_{\R ^2 }\phi_\ell}, & \ell \in \Cs \setminus \Is ^+,} \\
\varphi_1(x) &:= \sum_{\ell \in \Lambda} \varphi_1(\ell) \phi_\ell(x), \quad \forall x \in \R ^2,\\
\varphi_2(\ell) &:= \cases{
\frac{\int_{\R ^2 }(\varphi_h-\varphi_1) \phi_\ell^{\rm PU}}{\int_{\R ^2} \phi_\ell}, & \ell \in \Cs \setminus \Is^+,\\
0, & \ell \in \As \cup \Is \cup \Is ^+,}\\
\varphi_2(x) &:= \sum_{\ell \in \L} \varphi_2(\ell) \phi_\ell(x), \quad \forall x \in \R ^2.
\end{aligned}
\end{equation*}

Observe that $\varphi_h$ and $\varphi$ both are supported on a finite domain,
hence we can use Theorem 3.1 in \cite{carstensen} to conclude that
	\begin{equation*}
	\|\D \varphi\|_{L^2(\R ^2)} \lesssim \|\D \varphi_h\|_{L^2(\R ^2)}, \quad \forall \varphi_h \in \mathcal{U}_h .
	\end{equation*}

Let $g : = -{\rm div}\, [\pp _{\mF}W(\D \tilde{u}^{\rm a})]$. Then
\begin{equation*}
  \del_1
  = \int_{\Omega^\c} g \cdot (\varphi_h-\varphi) \dx = \int_{\Omega^\c} g \cdot \left((\varphi_h-\varphi_1)- \varphi_2\right) \dx
\end{equation*}
Since $\varphi_2$ is a piecewise-linear quasi-interpolant of $\varphi_h - \varphi_1$ as defined in \cite{carstensen}, a direct consequence of Theorem 3.1 in \cite{carstensen} is that there exists $C>0$ such that, recalling $\Omega_h^\a : = \bigcup \mathcal{T}_h^\a$,
\begin{equation*}
\del_1 \le C\|\D (\varphi_h-\varphi_1) \|_{L^2(\R ^2 \setminus \Omega_h^\a)}\left( \sum_{\ell \in \Cs \cap \Omega_h} d_\ell ^2 \int_{w_\ell} \phi^{\rm PU}_\ell |g - \<g\>_\ell|^2 \dx\right)^{1/2},
\end{equation*}
where $w_\ell : = \supp (\phi_\ell)$,
$\<g \>_\ell := 1/|w_\ell| \int_{w_\ell} g(x)\dx$ and
$d_\ell :={\rm diam}(w_\ell) = 1$. With the sharp Poincar\'{e} constant derived
by \cite{Acosta2003} , we have
\begin{equation*}
\int_{w_\ell } \phi^{\rm PU}_\ell |g - \<g\>_\ell|^2 \dx \le \int_{w_\ell} |g-\<g\>_\ell|^2 \dx\le \tfrac{1}{4} d^2_\ell \|\D g\|^2_{L^2(w_\ell)}.
\end{equation*}
On the other hand, $\varphi_1$ is a standard quasi-interpolant of $\varphi_h$ in $ \bigcup{\mathcal{T}_h^c}$, which implies that there exists $C' >0$ such that
\begin{equation}
  \label{eq:stab_Pi*}
  \|\D (\varphi_h - \varphi_1)\|_{L^2(\R ^2 \setminus \Omega_h^\a)}  \le C' \|\D \varphi_h\|_{L^2(\R ^2 \setminus \Omega_h^\a)}.
\end{equation}

Due to the fact that $d_\ell = 1$ and that each point in
$\R^2 \setminus \Omega^\a_h$ is covered by at most three $w_\ell$, we have
\begin{align}
\notag
  \del_1 &\le C \max_\ell d_\ell^2 \|\D g\|_{L^2(\R ^2 \setminus \Omega_h^\a)}\|\D \varphi_h\|_{L^2(\R ^2 \setminus \Omega_h^\a)}  \\
  \label{eq:final-estimate-delta1}
  &\le C\left(M_2\|\D^3 \tilde{u}^\a\|_{L^2(\R ^2 \setminus \Omega_h^\a)}+M_3\|\D^2 \tilde{u}^\a \|^2_{L^4(\R ^2 \setminus \Omega_h^\a)}\right)\|\D \varphi_h\|_{L^2(\R ^2 \setminus \Omega_h^\a)},
\end{align}
where we used the following estimate, for some $c>0$,
\begin{equation*}
\begin{aligned}
\|\D g  \|_{L^2(\Omega_h)} &= \|\D {\rm div} [\pp _\mF W (\D \tilde{u}^\a)]\|_{L^2(\R ^2 \setminus \Omega_h^\a)}  \\
&= \|\D \left(\pp_\mF ^2 W (\D \tilde{u}^\a) \D ^2 \tilde{u}^\a \right)\|_{L^2(\R ^2 \setminus \Omega_h^\a)} \\
&= \left\| \pp_\mF ^2 W (\D \tilde{u}^\a) \D ^3 \tilde{u}^\a +  \pp_\mF ^3 W (\D \tilde{u}^\a) \left(\D ^2 \tilde{u}^\a \right)^2\right\|_{L^2(\R ^2 \setminus \Omega_h^\a)}  \\
&\le c \left(M_2\|\D^3 \tilde{u}^\a\|_{L^2(\R ^2 \setminus \Omega_h^\a)}+M_3\|\D^2 \tilde{u}^\a \|^2_{L^4(\R ^2 \setminus \Omega_h^\a)}\right),
\end{aligned}
\end{equation*}
employing the global bounds \eqref{def:M_2} and \eqref{def:M_3}.
This completes the estimate for $\delta_1$.

\subsection{Estimation of $\delta_2$}
Recall that
\begin{equation*}
\delta_2 : = \int_{\Omega^\c} (\partial_\mF W(\D \Pi_h u^\a) - \partial_\mathsf{F} W(\nabla \tilde{u}^\a)):\nabla \varphi_h.
\end{equation*}

We start with estimating the best approximation error.

\begin{lemma}\label{lem:ba}
  Let $T\in \mathcal{T}^c_h$, $u \in \dot{\mathcal{U}}^{1,2}$ and
  $v \in W^{3,2}(\R ^2)$. Then we have the following estimates.
\begin{enumerate}
	\item[(a)] Denote $h_T: = {\rm diam}(T)$, then
	\begin{equation*}
	\|\nabla v- \nabla I^2_h v\|_{L^2(T)}\lesssim h^2_T\|\nabla^{3} v\|_{L^2(T)}.
	\end{equation*}
      \item[(b)] There exists a constant $C>0$ such that, for any domain
        $S \supset \mathcal{B}_{N}$,
	\begin{equation*}
	\|\nabla \mathcal{L}u - \nabla \tilde{u}\|_{L^2(S)} \leq C \|\nabla \tilde{u}\|_{L^2\left(S\setminus \mathcal{B}_{N/2}\right)} ,
	\end{equation*}
	where $\mathcal{L}$ is the cut-off function defined by \eqref{def:cut-off}.
      \item[(c)] Furthermore, we have the best approximation error
        estimate \begin{equation} \label{eq:best_approx}
        \begin{aligned}
          \|\nabla \Pi_h u - \nabla
          \tilde{u}\|_{L^2(\Omega^\c )} \lesssim &\|h^2 \nabla^{3} \tilde{u}^\a \|_{L^2(\Omega^\c _h)} +\|\nabla \tilde{u}^\a \|_{L^2\left(\R ^2 \setminus \mathcal{B}_{N/2}\right)}\\
          & +N^{-1} \| h^2\nabla^2 \tilu \|_{L^2(\mathcal{B}_{N} \setminus
              \mathcal{B}_{N/2})},
	        \end{aligned}
	\end{equation}
	where $h(x) := {\rm diam}(T)$ with $x \in T$.
	\end{enumerate}
\end{lemma}

\begin{proof}
  Recall the uniform shape regularity assumption \eqref{eq:unif-shape-reg}.

  Part (a) follows directly from the Bramble--Hilbert Lemma.

  For Part (b), we use a variation of \cite{2012-ARXIV-ellRd} Theorem
  2.1. Applying Poincar\'{e}'s inequality gives
  \begin{equation*}
    \begin{aligned}
      \|\nabla \mathcal{L}(u)-\nabla \tilde{u}\|_{L^2(S )}
      &=
      \big\| N^{-1}\mu'(\tilde{u}-a)
      + (\mu-1)\nabla \tilde{u}\big\|_{L^2(S )}\\
      &\leq N^{-1} C_\mu \| \tilde{u} - a \|_{L^2(S )}
      + \| (1-\mu)\nabla \tilde{u} \|_{L^2(S \setminus B_{N/2} )} \\
      &\le C_pC_\mu \|\nabla \tilde{u}\|_{L^2(\mathcal{B}_N\setminus\mathcal{B}_{N/2})}+  \|(1-\mu)\nabla \tilde{u} \|_{L^2(S \setminus \mathcal{B}_{N/2})}\\
      &\le C \|\nabla \tilde{u}\|_{L^2(S \setminus \mathcal{B}_{N/2})}.
    \end{aligned}
  \end{equation*}

For Part (c), we combine Part (a) and (b), that is
\begin{equation*}
\begin{aligned}
  \|\nabla \Pi_h u - \nabla \tilde{u}\|_{L^2(\Omega^\c)} &\le \|(I^2_h \circ \mathcal{L})( u) -\mathcal{L}(u) \|_{L^2(\Omega^\c)}+\| \mathcal{L}(u) -\tilde{u}\|_{L^2(\Omega^\c)} \\
  & \lesssim \|h^2\nabla^{3}\mathcal{L}(u) \|_{L^2(\Omega^\c)} +\|\nabla \tilde{u}\|_{L^2(\Omega^\c \setminus \mathcal{B}_{N/2})}\\
  & = \left\|h^2\sum_{n=0}^{3}\frac{1}{N^n}\nabla^{n}\mu \nabla^{3-n} (\tilde{u}-a) \right\|_{L^2(\mathcal{B}_N \setminus \Omega^\a)} +\|\nabla \tilde{u}\|_{L^2(\R ^2 \setminus \mathcal{B}_{N/2})}\\
  & \lesssim \|h^2 \nabla^{3} \tilde{u}\|_{L^2(\Omega^\c _h)} +\|\nabla \tilde{u}\|_{L^2(\R ^2 \setminus \mathcal{B}_{N/2})} + \frac{1}{N} \| h^2\nabla^2 \tilu \|_{L^2(\mathcal{B}_{N} \setminus B_{N/2})}.
\end{aligned}
\end{equation*}
The last line only contains the terms with $n = 0, 1$. The term for $n=2$ is
$N^{-2} \| h^2 \nabla \tilu \|_{L^2(\mathcal{B}_{N}\setminus B_{N/2})}$, but
since $N^{-2} h^2 \lesssim 1$ this is absorved into
$\| \nabla \tilu \|_{L^2(\R^2 \setminus \mathcal{B}_{N/2})}$.  For $n=3$, using
Poincar\'{e}'s inequality a similar argument applies.
\end{proof}

The estimate for $\del_2$ is now a consequence of the best approximation error
estimate:
\begin{align}
\notag
\del _2 & \le \|\partial_\mF W(\D \Pi_h u^\a) - \partial_\mathsf{F} W(\nabla \tilde{u}^\a)\|_{L^2(\Omega^\c)} \|\D \varphi_h \|_{L^2(\Omega^\c)} \\
\notag
&\le M_2 \|\nabla \Pi_h u^\a  - \nabla \tilde{u}^\a\|_{L^2(\Omega^\c)} \|\D \varphi_h\|_{L^2(\Omega^\c)}\\
\label{eq:delta_2}
& \lesssim \left(\|h^2 \nabla^{3} \tilde{u}^\a \|_{L^2(\Omega^\c _h)} +\|\nabla \tilde{u}^\a \|_{L^2\left(\R ^2 \setminus \mathcal{B}_{N/2}\right)}  + N^{-1} \| h^2\nabla^2 \tilu \|_{L^2(\mathcal{B}_{N} \setminus B_{N/2})}  \right) \|\D \varphi_h\|_{L^2(\Omega^\c)}.
\end{align}

\subsection{Estimation of $\delta_3$}
\label{sec:est-del3}
Recall that
\begin{displaymath}
  \delta_3 =  \int_{\Omega^\c} \big[\partial_{\sf F} W(\nabla \tilde{u}^\a)
  - \partial_{\sf F} W(\nabla {u}^\a)\big]: \nabla \varphi,
\end{displaymath}
where $\varphi$ is a lattice function with compact support and $\nabla \varphi$
denotes the gradient of its piecwise linear interpolant. To estimate this term
we observe that $u^\a$ can be interpreted as the P1 nodal interpolant of
$\tilu^\a$. Although this indicates a first-order estimate only, we can exploit
mesh regularity to obtain a second-order superconvergence estimate.

To that end, we rewrite the integral domain as a summation of elements. Let
$\mathring{E}$ be the union of edges that are shared by two continuum elements,
and $\omega_e$ be the union of said elements, i.e.,
\begin{equation*}
\begin{aligned}
\mathring{E} : &=\{ e = T_1 \cap T_2 \sep T_1, T_2 \in \T _ \Cs \}.\\
\omega_e : &= T_1 \cup T_2,\quad \text{where } T_1\cap T_2 = e.
\end{aligned}
\end{equation*}
Recall that $W(\mF) \equiv \frac{1}{\Omega_0} V(\mF \cdot \mathbf{a})$. Observe
that for a pair of $T_1,T_2$ sharing a common edge $e$ which has the direction
of $a_j$, $\D_{a_j} \varphi (T_1) = \D_{a_j} \varphi(T_2)$, which allows us to
re-group integration over elements as integration of patches $\omega_e$ {\em
  except} for elements near the interface. After simplifying the notation by
writing $\tilde{V}_j := \pp_j V(\D\tilu \cdot {\bfa })$ and
$V_j := \pp_j V(\D u \cdot {\bf a})$, we can rewrite $\delta_3$ as follows:

\begin{equation*}
\begin{aligned}
  \delta_3
  & = \frac{1}{\Omega_0}\sum_{T\in\T_\Cs \cup \T_\Is} \sum_{j=1}^{6}
  \int_{T \cap \Omega^\c} \big[ \tilde{V}_j - V_j \big] \cdot \D_{a_j} \varphi(T) \\
  & = \frac{1}{\Omega_0} \sum_{j = 1}^6
  \sum_{\substack{e \in \mathring{E}_j}}
  \int_{\omega_e} \big[ \tilde{V}_j - V_j \big] \cdot \D_{a_j} \varphi
  \\ &\qquad
  + \frac{1}{\Omega_0} \sum_{T\in \T_\Cs \cup \T_\Is} \sum_{j=1}^{6}
  c_{T,j} \int_{T \cap \Omega^\c} \big[ \tilde{V}_j - V_j \big] \cdot \D_{a_j} \varphi(T) \\
  &=: \tau_1 + \tau_2,
\end{aligned}
\end{equation*}
where $\mathring{E}_j: = \{e\in \mathring{E}\sep e\text{ is in the direction of } a_j\}$ and
$c_{T,j}$ is defined as follows,
\begin{equation*}
c_{T,j} = \cases{ 0, &  \exists e \in \mathring{E}_j \cap T ,\\
	1, & \text{otherwise}.
}
\end{equation*}
Observe that for $T \in \T_\Cs$, $c_{T,j}$ is only non-zero near the interface. So we have
\begin{equation}
  \label{eq:tau2-estimate}
\tau_2\le \frac{1}{\Omega_0} \int_{\Omega^i_+} M_2 |\D \tilde{u}^\a - \D u^\a | \, |\D \varphi| \lesssim \| \D ^2 \tilde{u}^\a \|_{L^2(\Omega^i_+)} \|\D \varphi\|_{L^2(\Omega^i_+)},
\end{equation}
where
$\Omega^i_+ : = \bigcup\{ T \in \T_\Cs \sep \text{dist} (T, \Omega^i) \le 1/2
\}$.
Note that the second-order error $\D ^2 \tilde{u}^\a$ results from the fact that
$u^\a$ is a piecewise linear nodal interpolant of $\tilde{u}^\a$ on a uniform
mesh.

To estimate $\tau_1$, we employ the following second-order mid-point estimate.

\begin{lemma}\label{lem:mprule}
  Suppose $f\in W^{2,\infty}(T_1 \cup T_2;\mathbb{R})$ where
  $T_1,T_2 \in \mathcal{T}$ such that they share an edge $e$ and let $m_e$
  be the mid-point of $e$, then
  \begin{equation*}
    \left|\int_{T_1 \cup T_2} f(\xi) - f(m_e) \, {\rm d\xi} \right| \lesssim \|\nabla^2 f\|_{L^\infty(T_1\cup T_2)}.
  \end{equation*}
\end{lemma}


Then we can write
\begin{equation}
  \label{eq:tau1-step1}
  \tau_1
  = \frac{1}{\Omega_0} \sum_{j = 1}^6
  \sum_{\substack{e\in \mathring{E}_j}}
  \int_{\omega_e }\Big[ (\tilde{V}_j  - \tilde{V}_j(m_e))
  - ( V_j -\tilde{V}_j(m_e)) \Big] \cdot \D_{a_j} \varphi.
\end{equation}
By Lemma \ref{lem:mprule} we have
\begin{align}
\notag
  \left|\int_{\omega_e } \big(\tilde{V}_j  - \tilde{V}_j(m_e)\big) \right|
  &\lesssim  \|\D^2 \pp_j V(\D \tilde{u}^\a \cdot \mathbf{a}))\|_{L^\infty(\omega_e)}  \\
\notag
& \lesssim \left(M_3\|\D ^2 \tilde{u}^\a \|^2_{L^\infty(\omega_e)} + M_2\|\D ^3 \tilde{u}^\a \|_{L^\infty(\omega_e)} \right)\\
\label{eq: V_mid}
& \lesssim \|\D ^2 \tilde{u}^\a \|^2_{L^4(\omega_e)} + \|\D ^3 \tilde{u}^\a \|_{L^2(\omega_e)},
\end{align}
where the last line comes from the fact that $\tilde{u}^\a $ is a polynomial of
degree $5$ on each $T$, hence on each patch $\omega_e$ the norms are equivalent.

On the other hand,  for $i = 1,...6$ we denote $\nu_{T,i}$ and $\nu_{T,i'}$ as the vertices of $T$ with $\nu_{T,i} + a_i = \nu_{T,i'}$. Then on $T \supset e$, we have, using Taylor expansion,
\begin{align*}
  \D u^\a |_T \cdot a_i - \D \tilde{u}^\a (m_e)\cdot a_i
  &= \tilu ^\a (\nu_{T,i'}) - \tilu^\a  (\nu_{T,i}) -  \D \tilu^\a(m_e) \cdot a_i
      = \tau_e, \qquad
\end{align*}
$\text{where } |\tau_e| \lesssim \|\D ^3\tilu ^\a\|_{L^\infty(\omega_e)}$.
Then for $T_1$ and $T_2$ with $T_1 \cap T_2 = e = [\nu_{T,i}, \nu_{T,i'}]$, we
have
\begin{align*}
[\D u^\a (T_1) \cdot a_i - \D \tilde{u}^\a (m_e)\cdot a_i ] +[\D u^\a (T_2) \cdot a_i - \D \tilde{u}^\a (m_e)\cdot a_i ] = 2\tau_e.
\end{align*}
(See also Figure \ref{fig:T_1_T_2}.)
\begin{figure}
	\centering
	\includegraphics[width=0.4\linewidth]{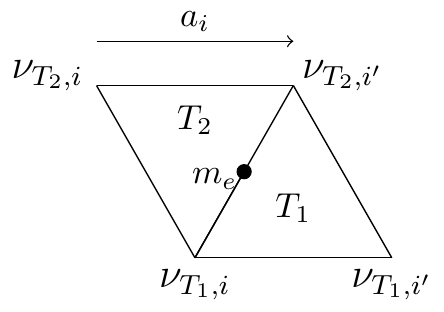}
	\caption{\small{  }}
	\label{fig:T_1_T_2}
\end{figure}
Hence, we can estimate
\begin{equation*}
\begin{aligned}
&\int_{\omega_e } {V}_j  - \tilde{V}_j(m_e)  \\
=&|T_1|(V_j|_{T_1} - \tilde{V}_j(m_e)) + |T_2|(V_j|_{T_2} - \tilde{V}_j(m_e)) \\[1mm]
=& |T_1|\Bigg\{ \sum_{i=1}^6 \pp_{j,i} V(\D \tilu ^a (m_e) \cdot \mathbf{a} ) \Big[
\D u^\a |_{T_1} \cdot a_i - \D \tilde{u}^\a (m_e)\cdot a_i \\[-3mm]
& \hspace{5.5cm} + \D u^\a |_{T_2} \cdot a_i - \D \tilde{u}^\a (m_e)\cdot a_i \Big]  \\
& \hspace{2cm}+ \mathcal{O}\big( M_3 \|\D^2 \tilde{u}^\a \|^2_{L^\infty(\omega_e)} \big)
   \Bigg\} \\
&\lesssim M_3 \|\D ^3\tilu ^\a\|_{L^\infty(\omega_e)}
  + M_2 \|\D ^2 \tilde{u}^\a \|^2_{L^\infty(\omega_2)}.
\end{aligned}
\end{equation*}

Combining this estimate with \eqref{eq: V_mid}, we have
\begin{align*}
\tau_1 \lesssim \Big\{\|\D ^2 \tilde{u}^\a \|^2_{L^4(\Omega^\c)} + \|\D ^3 \tilde{u}^\a \|_{L^2(\Omega^\c)}\Big\} \|\D \varphi \|_{L^2(\Omega^\c)}.
\end{align*}

Finally, combining the last estimate with \eqref{eq:tau2-estimate} we obtain
\begin{equation}
  \label{eq:delta_3-final}
  \delta_3 \lesssim \Big\{\| \D ^2 \tilde{u}^\a \|_{L^2(\Omega^i_+)} + \|\D ^2 \tilde{u}^\a \|^2_{L^4(\Omega^\c)} + \|\D ^3 \tilde{u}^\a \|_{L^2(\Omega^\c)}\Big\} \|\D \varphi \|_{L^2(\Omega^\c)}.
\end{equation}

\subsection{Estimation of $\delta_4$}
\label{sect:delta_4}
We observe that $\delta_4$ requires the estimation of pure modelling errors
regardless of the choice of finite element approximation or domain
truncation. This term was the main focus of \cite{PRE-ac.2dcorners}, where the
following result was proven.

\begin{theorem}[Theorem 5.1 \cite{PRE-ac.2dcorners}]
  Let $u : \L \to \R^m$ and let $\varphi : \L \to \R^m$ with compact support,
  then
  \begin{align*}
    & \hspace{-1cm}
      \< \delta \E^{\rm g23}(u^\a) -  \delta \E^\a (u^\a), \varphi \> \\
    &\lesssim
    \Big(
    M_2\|D^2 u^\a\|_{\ell^2(\Is ^{\rm ext})}
    + M_2 \|D^3 u^\a \|_{\ell^2 (\Cs )}
    + M_3 \|D^2 u^\a \|^2_{\ell^4(\Cs)} \Big) \|D \varphi\|_{\ell ^2(\Lambda  \setminus \As )},
  \end{align*}
  where
  $\Is^{\rm ext }: = \{\ell \in \Lambda \sep \text{\rm dist}(\ell, \Is) \le
  1\}$.
\end{theorem}

By the construction of the smooth interpolant $\tilu$ in Lemma
\ref{lem:smooth_int} we therefore conclude that
\begin{equation}
  \label{eq:del4_final}
  \delta_4 \lesssim (\|\D ^2 \tilde{u}^\a \|_{L^2(\Omega ^\i)}+\|\D ^3 \tilde{u}^\a \|_{L^2(\Omega^\c)} + \|\D ^2 \tilde{u}^\a\|^2_{L^4(\Omega^\c)}) \|\D \varphi\|_{L^2(\R ^2\setminus \Omega^\a )}.
\end{equation}

\subsection{Proof of Theorem \ref{thm:consist}}
Recall from \eqref{eq:eta-int+eta-ext} the splitting of the consistency error
into $\eta_{\rm ext}$ and $\eta_{\rm int}$. From the definition of $\varphi$ in
\eqref{def: varphi_intp} it follows that $\eta_{\rm ext} = 0$.

In \eqref{eq:decomp} the term $\eta_{\rm int}$ is further split into
$\delta_1, \dots, \delta_4$ which are respectively estimated in
\eqref{eq:final-estimate-delta1}, \eqref{eq:delta_2}, \eqref{eq:delta_3-final}
and \eqref{eq:del4_final}. Combining these four estimates, the stated result
\eqref{eq:consist} follows.

\subsection{Proof of the work estimate  \eqref{eq:consis_N} } \label{sec:optimal_mesh}
A key aspect of our analysis is the optimisation of the approximation
parameters: First, we determine a mesh size $h$ so that the finite element error
is balanced with the modelling error. Secondly, the domain radius $N$ and the
atomistic radius $K$ will be balanced. Finally, in order to compare the
efficiency against different methods, we will express the convergence rate of
the total error in terms of numbers of degree of freedom only.

We first estimate the decay rate of each term in the consistency estimate
\eqref{eq:consist}.  Recall that Corollary \ref{thm:decay} implies
$|\D^j \tilde{u}^\a (x)|\lesssim |x|^{-1-j}$. Hence, we can estimate the
interface error by
\begin{equation}\label{eq:op_rate}
\|\D^2 \tilde{u}^\a\|_{L^2(\Omega^{\i})} \lesssim\left( \int_{\Omega^\i}|x|^{-6}\right)^{\frac{1}{2}}\lesssim (K \cdot K^{-6})^{\frac{1}{2}} \lesssim K ^{ - 5/2}.
\end{equation}
Similarly, we have
\begin{align}
\notag
\|\D ^3 \tilde{u}^\a \|_{L^2(\Omega^\c)} &\lesssim  \left(\int_{\Omega^\c} |x|^{-8}\dx \right)^{\frac{1}{2}} \lesssim \left( \int_K^\infty r \cdot r^{-8}\dr\right)^{\frac{1}{2}} \lesssim K^{-3},\\
\notag
\|\D ^2 \tilde{u}^\a \|^2_{L^4(\Omega^\c)} & \lesssim \left(\int_{\Omega^\c} |x|^{-12}\dx \right)^{\frac{1}{2}} \lesssim \left( \int_K^\infty r \cdot r^{-12}\dr\right)^{\frac{1}{2}} \lesssim K^{-5}, \\
\label{eq:far-field}
\|\D \tilde{u}^\a \|_{L^2(\R \setminus \mathcal{B}_{N/2})} &\lesssim \left(\int_{\R ^2 \setminus \mathcal{B}_{N/2} } |x|^{-4}\dx \right)^{\frac{1}{2}} \lesssim \left( \int_{N/2}^\infty r \cdot r^{-4}\dr\right)^{\frac{1}{2}} \lesssim N^{-1}.
\end{align}
We observe that the interface term \eqref{eq:op_rate} dominates the consistency
error. Balancing this with the far-field term \eqref{eq:far-field} gives
\begin{equation*}
N^{-1} \approx  K ^{ - 5/2}, \quad \text{i.e.,} \quad N \approx K^{5/2}.
\end{equation*}
To determine the mesh size $h$, we write $h(x) := \left( \frac{|x|}{K}\right)^\beta$. Then we have
\begin{equation*}
\begin{aligned}
\|h^2 \D ^3 \tilde{u}^\a \|_{L^2(\Omega_h^\c)}& \lesssim \left(\int_{\Omega_h^\c}\frac{|x|^{4\beta}}{K^{4\beta}}|x|^{-8} \dx \right)^{1/2}\\
& = \frac{1}{K^{2\beta}}\left(\int_K^N r \cdot r^{4\beta-8} \dr \right) ^{1/2}\\
& = \frac{1}{K^{2\beta}} \left( \left[ r^{4\beta-6}\right]_{r=K}^{r = N} \right)^{1/2}\\
& \approx K^{-3}, \quad \text{provided that }4\beta-6 <0.
\end{aligned}
\end{equation*}

The final remaining term is
\begin{align*}
  N^{-1} \| h^2\nabla^2 \tilu ^\a \|_{L^2(\mathcal{B}_{N} \setminus B_{N/2})}
  &\lesssim N^{-1} \left(\int_{\mathcal{B}_{N} \setminus  B_{N/2}}
    \frac{|x|^{4\beta}}{K^{4\beta}}|x|^{-6} \dx \right)^{1/2}\\
 &\lesssim N^{2\beta - 3} K^{-2\beta}
   \lesssim K^{3\beta - \frac{15}{2}}.
\end{align*}
Since we chose $\beta < 3/2$, it follows that $K^{-3}$ dominates
$K^{3\beta - \frac{15}{2}}$.

Therefore, the optimal rate for the finite element coarsening is $K^{-3}$ and to
attain it we must choose
\begin{equation*}
  h(x) \approx \left( \frac{|x|}{K}\right)^\beta, \quad \text{where } \beta < \frac{3}{2}.
\end{equation*}

Finally, we estimate the relationship between the number of degrees of freedom
$\mathcal{N}_h$ and the atomistic radius $K$. It is easy to see that the number
of degrees of freedom in the atomistic domain satisfies
$\mathcal{N}_\a \approx K^2$.  Next, one can estimate the degrees of freedom in
the continuum domain $\mathcal{N}_\c$ by considering each hexagonal layer of the
mesh. On each layer with radius $r$,
$\mathcal{N}_{\rm layer} \approx \frac{r}{h(r)}$. Summing over all layers in the
continuum region gives
\begin{equation*}
\begin{aligned}
  \mathcal{N}_\c &\approx \sum_{\text{\rm layers in } \Omega_\c }\left(h \frac{1}{h}\right) \frac{r}{h} \\
  & \approx \int_K^N \frac{r}{h(r)^2} \dr \\
  & \approx \int_K^N r^{1-2\beta} K^{2\beta}\dr \\
  & \approx (-N^{2-2\beta}+K^{2-2\beta}) K^{2\beta}, \qquad \text{provided that } 2-2\beta < 0,\\
  & \approx K^2.
\end{aligned}
\end{equation*}
Therefore, we deduce that the mesh grading should satisfy
$1 < \beta < \frac{3}{2}$ to obtain the optimal cost/accuracy ratio for the
error in the energy-norm, $K^{-5/2} \approx \mathcal{N}_h^{-5/4}$. The table in
\S~\ref{sec:optim-params} summarises the derivation of this section.


\section{Proof of Theorem \ref{theo:main}}\label{sect:main_proof}
\subsection{Existence and error in energy norm}
We refer to the inverse function theorem, Theorem \ref{theo:inverse}. Let
$\delta \mathcal{G}_h : =\del \E^{\rm g23} - \del f$ and
$\bar{u}_h := \Pi_h u^\a$. We have already shown in Theorem \ref{thm:consist}
and Lemma \ref{lem:stab} that
\begin{align*}
\|\mathcal{G}_h(\bar{u}_h)\|_{\mathcal{U}_h ^*} &\le \eta, \\
\langle \delta \mathcal{G}_h(\bar{u}_h)v_h,v_h\rangle &\ge \gamma \|\nabla v_h\|^2_{L^2} \quad
\text{ for all }v_h \in \mathcal{U}_h,
\end{align*}
with $\eta = \eta_{\rm int} + \eta_{\rm ext}$ and
\begin{equation*}
\begin{aligned}
\eta_{\rm int}   &\lesssim  \|\D ^2 \tilde{u}^{\a}\|_{L^2(\Omega^{\rm i})}+\|\D ^3 \tilde{u}^\a \|_{L^2(\Omega^{\rm c})}
+ \|\D ^2 \tilde{u}^\a \|^2_{L^4(\Omega^{\rm c})}\\
& \qquad + \|h^2\D ^3 \tilde{u}^\a \|_{L^2(\Omega_h^{\rm c})} + \|\D  \tilde{u}^\a \|_{L^2(\R ^2 \setminus \mathcal{B}_{N/2})} + N^{-1} \| h^2\nabla^2 \tilu ^\a \|_{L^2(\mathcal{B}_{N} \setminus B_{N/2})} \\
&\lesssim \mathcal{N}_h^{-5/4}.
\end{aligned}
\end{equation*}

For $\eta_{\rm ext}$, recall that $\partial_{u(\ell)} f(u) = 0$ for all
	$|\ell| \geq R_f$, and that $K\ge R_f$. We have, on $\supp (\partial_{u(\ell)} f(u))$,  $\D \Pi_h u^\a = \D u^\a$ and $\D \varphi_h = \D \varphi$. Thus $\eta_{\rm ext} = 0$ and
	\begin{equation*}
	\eta  = \eta_{\rm int}\lesssim \mathcal{N}_h^{-5/4}.
	\end{equation*}

Using also the Lipscthiz bound from Lemma \ref{th:Lip} Theorem
\ref{theo:inverse} implies, for $K$ sufficiently large, that there exists a
strongly stable minimizer $u^{\rm g23}_h \in \mathcal{U}_h$ such that
 \begin{equation*}
\langle \del \E^{\rm g23}(u_h^{\rm g23} ),\varphi_h\rangle - \< \del f(u_h^{\rm g23}), \varphi_h\>=0, \quad \forall \varphi_h \in \mathcal{U}_h,
 \end{equation*}
 and
  \begin{equation*}
  \begin{aligned}
\|\D u^{\rm g23}_h - \D \Pi_h u^\a\|_{L^2}  &\le 2 \frac{\eta}{\gamma}\\
 & \lesssim \|\D ^2 \tilde{u}^{\a}\|_{L^2(\Omega^{\rm i})}+\|\D ^3 \tilde{u}^\a \|_{L^2(\Omega^{\rm c})}
 + \|\D ^2 \tilde{u}^\a \|^2_{L^4(\Omega^{\rm c})}\\
 &\quad  + \|h^2\D ^3 \tilde{u}^\a \|_{L^2(\Omega^{\rm c})} + \|\D  \tilde{u}^\a \|_{L^2(\R ^2 \setminus \mathcal{B}_{N/2})}\\
 &\lesssim  \mathcal{N}_h^{-5/4}.
\end{aligned}
\end{equation*}
Adding the best approximation error \eqref{eq:best_approx} gives
\begin{align*}
\|\D u^{\rm g23}_h - \D u^\a\|_{L^2} & \le \|\D u^{\rm g23}_h - \D \Pi_h u^\a\|_{L^2} + \|\D \Pi_h u^\a - \D u^\a \|_{L^2}\\
& \lesssim \mathcal{N}_h^{-5/4} + \|h^2
\nabla^{3} \tilde{u}\|_{L^2(\bigcup{\mathcal{T}_h^c})} +\|\nabla
\tilde{u}\|_{L^2\left(\R ^2 \setminus
	\mathcal{B}_{N/2}\right)}\\
& \lesssim \mathcal{N}_h^{-5/4}.
\end{align*}
This completes the proof of Theorem \ref{theo:main}.


\subsection{The energy error}\label{sec:energy_error}
In this section we prove the energy error estimates stated in Theorem
\ref{theo:main}. For the sake of notational simplicity we define
  $\E^\a_f := \E^\a - f$ and $\E^{\rm g23}_f := \E^{\rm g23} - f$.

First, we observe that
\begin{equation*}
\begin{aligned}
|\E^{\rm g23}_f(u^{\rm g23}_h) - \E^\a_f(u^\a)| & \le |\E^{\rm g23}_f(u_h^{\rm g23}) - \E^{\rm g23}_f(\Pi_h u^\a)| +  |\E^{\rm g23}_f(\Pi_h u^\a)- \E^\a_f(u^\a)| \\
& =: e_1 + e_2.
\end{aligned}
\end{equation*}
The first term can be estimated by \eqref{eq:u_error} and the fact that
$\<\del \E^{\rm g23}_f(u_h^{\rm g23}), \varphi_h\> = 0$ for all
$\varphi_h\in \mathcal{U}_h$:
\begin{align}
  \notag
  e_1 &\le \left| \<\del \E^{\rm g23}_f(u_h^{\rm g23}), \Pi_h u^\a - u_h^{\rm g23} \> \right| \\
\notag
  &\quad + \left|\int_0^1(1-t)\<\del^2 \E^{\rm g23}_f(u_h^{\rm g23} + t(\Pi_h u^\a - u_h^{\rm g23}))(\Pi_h u^\a - u_h^{\rm g23}), (\Pi_h u^\a - u_h^{\rm g23})\> \,dt\right|\\
\label{eq:est-e1}
& \lesssim \|\D \Pi_h u^\a - \D u_h^{\rm g23}\|^2_{L^2}
\lesssim K^{-5} \lesssim  \mathcal{N}_h^{-5/2}.
\end{align}
For the second term we use the fact that $\E^{\rm g23}(0) = \E^{\rm a}(0)$, and
hence $\E^{\rm g23}_f(0) = \E^{\rm a}_f(0)$, to estimate
\begin{align*}
  e_2
  &\le |\E^{\rm g23}_f(0)-\E^\a_f (0)|
    + \left| \int_0^1\< \del \E^{\rm g23}_f(t\Pi_h u^\a ), \Pi_h u^\a \> \, dt
    - \int_0^1\< \del \E^\a_f(t u^\a ), u^\a \> \, dt  \right|   \\
  &\le
    \left|\int_0^1 \< \del \E^{\rm g23}_f(t\Pi_h u^\a ), \Pi_h u^\a \>
    - \< \del \E^\a_f(t u^\a), v \> \, \dt \right|
    + \left| \int_0^1 \< \del \E^\a_f(t u^\a), v - u^\a \> \,\dt \right| \\
  &=: e_{21} + e_{22},
\end{align*}
where $v : \L \to \R^m$ is an arbitrary test function.

\subsubsection{Estimate for $e_{21}$}
To exploit the consistency error estimate we choose $v := \Pi_h^* \Pi_h u^\a$
defined in \eqref{def: varphi_intp}. In this case, we obtain
\begin{align*}
  e_{21}
  &\lesssim \int_0^1 \tilde{\eta}_{\rm int}(t) \, \dt \,
    \| \D \Pi_h u^\a \|_{L^2(\R^2 \setminus \Omega_h^\a)}, \qquad \text{where} \\
  \tilde{\eta}_{\rm int}(t)
  &=  \|\D ^2 t\tilde{u}^{\a}\|_{L^2(\Omega^{\rm i})} +
    \|\D ^3 t\tilde{u}^{\a} \|_{L^2(\Omega^{\rm c})}
    + \|\D ^2 t\tilde{u}^{\a} \|^2_{L^4(\Omega^{\rm c})}
    + \|h^2\D ^3 t\tilde{u}^{\a} \|_{L^2(\Omega_h^{\rm c})} \\
  & \qquad
    + \|\D  t\tilde{u}^{\a} \|_{L^2(\R ^2 \setminus \mathcal{B}_{N/2})}
    + N^{-1} \| h^2\nabla^2 t\tilu \|_{L^2(\mathcal{B}_{N} \setminus B_{N/2})} \\
  &\lesssim t K^{-5/2}.
\end{align*}
From Corollary \ref{thm:decay} and \ref{lem:ba} it follows that
$|\D \Pi_h v^\a(x)| \lesssim |x|^{-2}$ hence we can deduce that
\begin{equation}
  \label{eq:e21}
  e_{21} \lesssim K^{-5/2} K^{-1} = K^{-7/2} \lesssim \mathcal{N}_h^{-7/4}.
\end{equation}

\subsubsection{Estimate for $e_{22}$}
First we observe that by Trapezoidal rule, if $\zeta \in C^2(\R )$ and $ \zeta(0) = \zeta(1) = 0$, then we have， for some $\theta \in [0,1]$,
\begin{equation*}
  \int_0^1 \zeta(t) \dt =  -\frac{1}{12}\zeta''(\theta).
\end{equation*}
Let $\zeta (t) : = \< \del \E^\a_f(t u^\a), v - u^\a \>$.  Then $\zeta(1) = 0$
since $\del\E^\a_f(u^\a) = 0$ and $\zeta(0) = 0$ since $\del\E^\a(0) =0 $ and
$\pp_{u(\ell)} f(u) = 0$ outside defect core while $v = u^\a$ in the defect
core.

Having $e_{22} = \int_0^1 \zeta(t) \dt$ we obtain
\begin{align*}
e_{22} &\lesssim \delta^3 \E^\a(\theta u^\a)[u^\a, u^\a, v - u^\a]\\
& \lesssim M_3 \sum_{\ell \in \L \setminus \As  } |D u^\a (\ell) |^2 |D v(\ell ) - D u^\a(\ell) |\\
&\lesssim \int_{R^2 \setminus \Omega^\a}|\D \tilu ^\a |^2 |\D v - \D u^\a |,
\end{align*}
where we recall that $v := \Pi_h^* \Pi_h u^\a$. Using the stability
\eqref{eq:stab_Pi*} we obtain
\begin{align}
  \notag
  e_{22} &\lesssim \| \D \tilu^\a \|_{L^3(\R^2\setminus \Omega^\a)}^3
  + \| \D\tilu^\a \|_{L^4(\R^2 \setminus \Omega^\a)}^2 \| \D \Pi_h u^\a \|_{L^2(\R^2 \setminus \Omega^\a)}  \\
  \notag
  &\lesssim
    \int_{K}^\infty r r^{-6} \,\dr + \bigg( \int_K^\infty r r^{-8} \,\dr\bigg)^{1/2}
    \bigg( \int_K^\infty r r^{-4} \,\dr \bigg)^{1/2} \\
  \label{eq:est-e22}
  &\lesssim K^{-4} + K^{-3} K^{-1} = K^{-4}.
\end{align}

Combining \eqref{eq:est-e1}, \eqref{eq:e21} and \eqref{eq:est-e22} completes the
proof of the energy error estimate \eqref{eq:E_error} and therefore of our main
result, Theorem \ref{theo:main}.

\appendix

\bibliographystyle{plain}
\bibliography{qc.bib}
\end{document}